\documentclass[journal,twoside,web]{ieeecolor_nologo}
\usepackage{cite}
\usepackage{amsmath,amssymb,amsfonts}
\usepackage{algorithmic}
\usepackage{graphicx}
\usepackage{graphics}
\usepackage{tcolorbox}
\usepackage{url}
\usepackage{color}
\usepackage{hyperref}
\usepackage{textcomp}
\usepackage{mathtools, cuted}
\usepackage{lipsum}
\usepackage{stfloats}
\usepackage{multirow}
\usepackage{wrapfig}
\usepackage{caption}
\usepackage{subcaption}

\definecolor{subsectioncolor}{RGB}{0,0,128} % Dark blue

\usepackage{amsmath,amssymb,euscript,psfrag,latexsym,graphicx}
\usepackage{bbm,color,amstext,wasysym,cuted,mathtools, cite}

\usepackage[normalem]{ulem}
\graphicspath{{./},{./figures/}}
\usepackage{dsfont}
\usepackage{ulem}

\usepackage{algorithm}
\usepackage{algorithmic}
\usepackage[algo2e,linesnumbered,ruled]{algorithm2e}

	\usepackage[utf8]{inputenc}

\usepackage{mathtools,halloweenmath}

\newtheorem{theorem}{Theorem}
\newtheorem{definition}{Definition}

\newtheorem{assumption}{Assumption}
\newtheorem{remark}[theorem]{Remark}

\newtheorem{corollary}{Corollary}

\usepackage[dvipsnames]{xcolor}

\newtheorem{proposition}{Proposition}
\providecommand{\qed}{\hfill\QED}
\begin{document}

% \title{Koopman Supereigenfunctions and Positive Systems for Data-driven  Modeling, Control, and Safety}

% 

\title{Data-to-Certificates (D2C): Koopman Supereigenfunctions for Stability, Safety, and Control}

\author{ Umesh Vaidya, {\it IEEE Senior Member}
\thanks{
Financial support from of NSF CMMI 2031573 is greatly acknowledged. The author is with the Department of Mechanical Engineering, Clemson University, Clemson SC, 29631. }}

\maketitle

% \begin{abstract}
% Traditional dynamical system models, including Koopman operator representations, are fundamentally equality-based, whereas many analysis and control tools rely on inequalities. This mismatch motivates representations that are intrinsically aligned with inequality-based reasoning. 
% In this paper, we introduce \emph{supereigenfunctions} of the Koopman operator as an inequality-based generalization of Koopman eigenfunctions. By relaxing the spectral equality to a one-sided bound, supereigenfunctions define exponential growth envelopes that characterize nonlinear dynamics without requiring explicit spectral structure. We establish their theoretical foundations and show that the associated rates recover intrinsic dynamical quantities such as Lyapunov exponents. 
% We develop two complementary constructions: a geometric approach based on the multiplicative ergodic theorem, and a resolvent-based formulation that enables data-driven computation directly from trajectory data. The resulting framework yields a unified, operator-theoretic representation for uncertainty propagation, stability certification, and stabilizing and safety-critical control design. Simulation results illustrate the effectiveness of the proposed approach.
% \end{abstract}

\begin{abstract}
Traditional dynamical system models, including Koopman operator representations, are fundamentally equality-based, whereas many analysis and control tools rely on inequalities. This mismatch motivates representations that are intrinsically aligned with certification tasks involved in the analysis and control synthesis problems. In this paper, we propose a \emph{data-to-certificates (D2C)} paradigm that bypasses explicit model construction and directly learns certificates from data.
We introduce \emph{supereigenfunctions} of the Koopman operator as an inequality-based generalization of eigenfunctions that define exponential growth envelopes encoding stability, safety, and uncertainty propagation, thereby serving as certificates for a range of control objectives. We establish their theoretical foundations and show that the associated rates recover intrinsic dynamical quantities such as Lyapunov exponents. Two complementary constructions are developed: a geometric approach based on the multiplicative ergodic theorem (MET), and a resolvent/Gramian formulation that enables computation directly from trajectory data.
The resulting framework yields certificates that can be used for stability and contraction analysis, as well as stabilizing and safety-critical control synthesis via convex quadratic programming-based optimization  program. Numerical examples demonstrate the effectiveness of the proposed data-driven certification approach for stabilization, contraction, and safe control design.

% Traditional dynamical system models, including Koopman operator representations, are fundamentally equality-based, whereas many analysis and control tools rely on inequalities. This mismatch motivates representations that are intrinsically aligned with inequality-based reasoning.
% In this paper, we propose a shift from model construction to direct certificate synthesis. We introduce \emph{supereigenfunctions} of the Koopman operator as an inequality-based generalization of eigenfunctions, which define exponential growth envelopes that directly encode stability, safety, and uncertainty propagation properties.
% We establish their theoretical foundations and show that the associated rates recover intrinsic dynamical quantities such as Lyapunov exponents. Two complementary constructions are developed: a geometric approach based on the multiplicative ergodic theorem (MET), and a resolvent-based formulation that enables computation directly from trajectory data.
% The resulting Koopman-to-Certificates (K2C) framework produces certificates that can be directly used for stability and contraction analysis, stabilizing and safety-critical control design, often leading to convex optimization-based controllers. Simulation results demonstrate the effectiveness of the proposed approach.
\end{abstract}

\section{Introduction}

Models of dynamical systems are not ends in themselves but  they are instruments for prediction, analysis, and control synthesis. From this perspective, the structure of a model should be dictated by the tasks it is meant to enable. Classical modeling paradigms including linear system representations and operator-theoretic approaches such as Koopman operator theory are predominantly \emph{equality-based}, describing system evolution through exact relations. In contrast, most of the powerful tools in systems and control, including Lyapunov stability theory, Hamilton--Jacobi (HJ) formulations for optimal control, and safety verification via barrier functions, are fundamentally \emph{inequality-based}. This reveals a structural mismatch between how models are constructed and how they are ultimately used.

Koopman operator theory provides a compelling framework for the analysis and control of nonlinear dynamical systems by lifting them into a linear space of observables \cite{koopman1931hamiltonian,mezic2005spectral,mezic2020spectrum,williams2015data,korda2018convergence,kaiser2017data,korda2020optimal,korda2018linear
,vaidya2025koopman,huang2018data,huang2018feedback}. In this setting, Koopman eigenfunctions define intrinsic coordinates in which the dynamics evolve linearly, enabling prediction and control. However, this framework is built on spectral equalities of the form $\mathcal{K}_f \phi = \lambda \phi$, which encode exact exponential evolution. While elegant, this representation faces well-known limitations: global eigenfunctions may fail to exist, the spectrum may be continuous or ill-defined, and for controlled systems the lifted dynamics become bilinear, complicating synthesis \cite{mezic2020spectrum,vaidya2025koopman,vaidya2026discovering}.

In contrast, synthesis methodologies such as Lyapunov theory, control Lyapunov and barrier functions, and HJ-based approaches are inherently inequality-driven. These methods characterize system behavior through bounds, envelopes, and certificates rather than exact representations. For example, Lyapunov functions certify stability via $\dot V \leq -\alpha(V)$, while HJ value functions encode optimality through viscosity inequalities \cite{Khalil_book,CrandallIshiiLions1992,evans2010partial}. These approaches emphasize \emph{certification} rather than \emph{representation}.

This paper proposes a unifying perspective that bridges these two paradigms. We introduce \emph{supereigenfunctions} of the Koopman operator, defined by relaxing the spectral equality to a one-sided inequality 
\[
\mathcal{K}_f \phi \leq \lambda \phi.
\]
This seemingly simple relaxation fundamentally changes the role of the representation. Instead of encoding exact trajectory evolution, supereigenfunctions define exponential growth envelopes that bound system behavior. As a result, they provide an operator-theoretic representation that is naturally compatible with Lyapunov functions, HJ value functions, and reachability and safety certificates \cite{Sontag_99,sontag1998mathematical,artstein1983stabilization,ames2017control,slotine1998contraction}.

The key conceptual shift of this work is that Koopman-based methods are no longer used to construct predictive models, but to construct \emph{certificates}. In this view, supereigenfunctions act as operator-theoretic analogues of Lyapunov and value functions, enabling stability, safety, and performance guarantees directly within the Koopman framework.

We develop the theory of supereigenfunctions along two complementary directions. First, we introduce \emph{directional supereigenfunctions}, defined on the tangent bundle, which capture anisotropic growth and contraction rates. These objects arise naturally from the multiplicative ergodic theorem (MET), with associated rates corresponding to Lyapunov exponents, thereby providing a geometric and dynamical interpretation of the framework \cite{oseledets1968multiplicative,ruelle1979ergodic}. Second, we develop a \emph{resolvent-based construction} that generates families of supereigenfunctions from user-defined probe functions. This construction admits trajectory-based representations and enables computation directly from data, without explicit knowledge of the system.

A key feature of the proposed framework is that supereigenfunctions are globally defined, nonnegative, and form a convex cone. 
% Unlike classical Koopman eigenfunctions, which may be complex-valued, non-smooth, or only locally defined, supereigenfunctions are well-suited for certification tasks \cite{mezic2020spectrum}. 
Moreover, their inherent non-uniqueness becomes a design degree of freedom: different choices correspond to different objectives, such as stabilization, safety, or uncertainty quantification. This enables a task-driven representation paradigm in which the choice of observables is aligned with the intended analysis or control objective. This is in contrast to the choice of observables or lifting function used for the finite dimensional approximation of the Koopman opertor or its spectrum where the focus in on discovering intrinsic coordinates or subspace for minimal linear representation of nonlinear system. 

Importantly, the proposed constructions admit data-driven implementations. In particular, resolvent-based approach allow supereigenfunctions to be computed directly from trajectory data and variational dynamics. This provides a scalable alternative to PDE-based or sum-of-squares \cite{Parrilothesis,moyalan2021sum,moyalan2023data,lasserre2008nonlinear,vaidya_stochastic_linearoperator} approaches and enables integration with modern learning-based pipelines, including high-fidelity simulators and world models \cite{joglekar2024multimodel_koopman,samak2024digital_twin_koopman,ha2018world}. As a result, the framework supports the direct extraction of provable guarantees from data.

It is important to distinguish the proposed supereigenfunction framework from classical control Lyapunov and barrier functions (CLFs/CBFs) as well as standard Koopman eigenfunctions. CLFs and CBFs are inequality-based certificates typically constructed heuristically or via optimization (e.g., sum-of-squares), and are often problem-specific. In contrast, supereigenfunctions arise directly from the Koopman operator, providing certificates that reflect the intrinsic system dynamics and can be computed from trajectory data.
While the construction differs fundamentally, the resulting control synthesis has a similar structure: both approaches lead to pointwise convex QPs enforcing inequality constraints. However, in the proposed framework, these constraints are derived from operator-theoretic constructions rather than imposed on candidate functions, yielding a data-driven and dynamically consistent alternative to classical certificate design.

The contributions of this paper are as follows. We introduce supereigenfunctions as an inequality-based generalization of Koopman eigenfunctions and establish their role as operator-theoretic certificates for nonlinear dynamical systems. We develop directional supereigenfunctions on the tangent bundle and show that they recover and generalize contraction-based stability and incremental analysis. We propose resolvent and trajectory-based constructions that enable computation of supereigenfunctions directly from data. We establish connections between the proposed framework and Lyapunov theory, contraction theory, and safety certification, providing a unified operator-theoretic perspective.  We show that supereigenfunction-based certificates lead to convex control synthesis via quadratic programs (QPs), unifying stabilization, safety, and contraction within a single optimization framework. We demonstrate the effectiveness of the approach through applications to uncertainty propagation, stabilizing control, safety filtering, and synchronization. Overall, this work establishes a new paradigm in which Koopman theory is used not for modeling system trajectories, but for constructing certificates that enable analysis and control of nonlinear systems.

%%%%%%%%%%%%%%%%%%%%%%%%%%%%%%%%%%%%%%%%%
\section{Preliminaries and Notations}
\label{sec:prelim}
%%%%%%%%%%%%%%%%%%%%%%%%%%%%%%%%%%%%%%%%%%%%%%%%%%%%%%%%%%%%%%%%%%%%%%%%%%%%%%
\subsection{Notations} Let \(\Omega\subset\mathbb R^n\) denote the state space and \(s_t:\Omega\to\Omega\) the flow generated by the dynamical system $
\dot x=f(x).
$
For vectors \(x,y\in\mathbb R^m\), the notation
$
x\preceq y
$
denotes componentwise ordering,
$
x_i\le y_i,
\; i=1,\ldots,m.
$
For symmetric matrices \(A,B\in\mathbb S^n\),
$
A\preceq B
$
denotes the Loewner partial order,
$
B-A\succeq0.
$
The usual scalar ordering is denoted by \(\le\).
The cone of nonnegative vectors is denoted by
$
\mathbb R^m_+
:=\{x\in\mathbb R^m:\ x_i\ge0,\ i=1,\ldots,m\}.
$
The sets of symmetric positive semidefinite and positive definite matrices are denoted by
$
\mathbb S_+^n$ and $
\mathbb S_{++}^n,
$
respectively.

\subsection{Koopman Operator}
This section reviews the Koopman operator framework for nonlinear autonomous dynamical systems and recalls Koopman eigenfunctions. Let $\Omega \subseteq \mathbb{R}^n$ be an open set and consider the system
\begin{equation}
\dot x = f(x), \qquad f \in C^1(\Omega,\mathbb{R}^n).
\label{eq:system}
\end{equation}
We assume forward completeness, i.e., for every $x \in \Omega$ the solution exists for all $t \ge 0$ and remains in $\Omega$. Let $s_t:\Omega \to \Omega$ denote the associated flow. Koopman theory studies the evolution of \emph{observables} $\phi:\Omega \to \mathbb{R}$ rather than states \cite{lasota1994chaos}. Let $\mathcal F$ be a function space such as $C^1(\Omega)$ or $L^2(\Omega,\mu)$. The Koopman operator $U_t:\mathcal F \to \mathcal F$ is defined 
\begin{equation}
(U_t \phi)(x) := \phi(s_t(x)).
\label{eq:koopman}
\end{equation}
The family $\{U_t\}_{t \ge 0}$ forms a strongly continuous semigroup with $U_{t+\tau} = U_t U_\tau$ and $U_0 = I$. Each $U_t$ is linear even when \eqref{eq:system} is nonlinear. The infinitesimal generator of $\{U_t\}$ is
\begin{equation}
\mathcal K_f \phi := \lim_{t \downarrow 0} \frac{U_t \phi - \phi}{t},
\label{eq:generator_def}
\end{equation}
whenever the limit exists. For $\phi \in C^1(\Omega)$,
\begin{equation}
(\mathcal K_f \phi)(x) := \nabla \phi(x) \cdot f(x),
\label{eq:generator_formula}
\end{equation}
i.e., $\mathcal K_f$ coincides with differentiation along trajectories.

\begin{definition}[Koopman eigenfunction]
A nonzero observable $\phi \in \mathcal F$ is a Koopman eigenfunction with eigenvalue $\lambda \in \mathbb C$ if
\begin{equation}
\mathcal K_f \phi = \lambda \phi.
\label{eq:eigen_def}
\end{equation}
Equivalently,
\begin{equation}
\phi(s_t(x)) = e^{\lambda t} \phi(x), \qquad \forall t \ge 0.
\label{eq:eigen_semigroup}
\end{equation}
\end{definition}

If $\{\phi_i\}_{i=1}^m$ are eigenfunctions with eigenvalues $\lambda_i$, then the lifted coordinates $z_i = \phi_i(x)$ satisfy $\dot z_i = \lambda_i z_i$, yielding linear, decoupled dynamics. For linear systems $\dot x = Ax$, if $w^\top A = \lambda w^\top$, then $\phi(x) = w^\top x$ is a Koopman eigenfunction with eigenvalue $\lambda$. For general nonlinear systems, smooth global eigenfunctions may fail to exist or may only be defined locally \cite{mezic2020spectrum}—for example, in the presence of multiple attractors or complex recurrent dynamics. Moreover, systems with continuous or ill-conditioned Koopman spectra pose significant challenges for spectral approximation.
.The equality \eqref{eq:eigen_def} can be fragile under perturbations and difficult to enforce from data. These limitations motivate inequality-based objects that preserve growth information while relaxing spectral equality.
%%%%%%%%%%%%%%%%%%%%%%%%%%%%%%%%%%%%%%%%%%%%%%%%%%%%%%%%%%%%%%%%%%%%%%%%%%%%%%

\section{Supereigenfunctions of the Koopman Operator}
\label{sec:supereigen}
%%%%%%%%%%%%%%%%%%%%%%%%%%%%%%%%%%%%%%%%%%%%%%%%%%%%%%%%%%%%%%%%%%%%%%%%%%%%%%

% Classical Koopman analysis is based on the spectral equality
% $\mathcal K_f \phi = \lambda \phi$, which characterizes exact exponential behavior.
% For many nonlinear systems, however, smooth global eigenfunctions may fail to exist—for example, in the presence of multiple attractors or complex recurrent dynamics. Moreover, systems with continuous or ill-conditioned Koopman spectra pose significant challenges for spectral approximation.

% These limitations motivate an inequality-based formulation that yields certified growth envelopes. 
From the perspective of positive operator theory, the Koopman semigroup preserves pointwise order, and supereigenfunctions arise naturally as cone-ordered subeigenfunctions of  positive operator \cite{krein1948linear,nussbaum1986collatz}.

% A useful way to view this relaxation is through the lens of 
% \emph{positive operator theory} \cite{meyer1991banach}. 
% Indeed, on an ordered function space the Koopman semigroup preserves 
% pointwise order, and supereigenfunctions are precisely \emph{cone-ordered 
% subeigenfunctions} of a positive operator in the sense of 
% Perron--Frobenius theory 
% \cite{krein1948linear,nussbaum1986collatz}.
%=============================================================================
\subsection{Ordered spaces and positivity}
%=============================================================================

Let $\mathcal F$ be a real Banach space of observables on $\Omega$ (e.g.,
$C(\Omega)$, $C^1(\Omega)$, or $L^p(\Omega,\mu)$), equipped with the positive cone
\[
\mathcal F_+ := \{\psi \in \mathcal F : \psi(x) \ge 0 \ \text{(a.e.) in } \Omega\}.
\]
This cone induces the partial order $\psi_1 \leq \psi_2$ iff $\psi_1 - \psi_2 \in \mathcal F_+$.
An operator $T:\mathcal F \to \mathcal F$ is \emph{positive} if $\phi \in \mathcal F_+$ implies $T\phi \in \mathcal F_+$. The Koopman semigroup $(U_t)_{t\ge 0}$ forms a positive semigroup on 
$\mathcal F$, preserving the cone $\mathcal F_+$ \cite{engel2000one,lasota1994chaos}. Indeed as  $U_t:\mathcal F\to\mathcal F$ defined by $
(U_t\psi)(x):=\psi(s_t(x))$
positivity follows:
\begin{equation}
\psi\in\mathcal F_+ \ \Longrightarrow\ U_t\psi\in\mathcal F_+,
\qquad \forall t\ge 0.
\label{eq:koopman_positive}
\end{equation}
In this ordered setting, one may study \emph{subeigenfunctions} (also called
cone subeigenvectors).
% defined by inequalities of the form $T\psi\preceq \rho\psi$,
% which motivates the following definition.

% For the Koopman semigroup $U_t$ defined by $(U_t \phi)(x) = \phi(s_t(x))$, positivity is immediate:
% \begin{equation}
% \phi \in \mathcal F_+ \ \Longrightarrow\ U_t \phi \in \mathcal F_+, \qquad \forall t \ge 0.
% \label{eq:koopman_positive}
% \end{equation}
% This ordered structure motivates the  subeigenfunctions. 

%=============================================================================
\subsection{Supereigenfunctions and envelope rates}
%=============================================================================

\begin{definition}[Supereigenfunction]
\label{def:supereigen}
A function $\varphi:\Omega \to \mathbb R_{\ge 0}$ is a supereigenfunction of $\mathcal K_f$ if there exists $\lambda \in \mathbb R$ such that
\begin{equation}
(\mathcal K_f \varphi)(x) \le \lambda \varphi(x), \qquad \forall x \in \Omega.
\label{eq:supereigen}
\end{equation}
Any such $\lambda$ is called an admissible supereigenvalue.
\end{definition}
Along the trajectories of the system $s_t(x)$, we have
\[
\frac{d}{dt}\varphi(x(t)) \le \lambda \varphi(x(t)),
\]
so $\varphi$ is bounded by an exponential envelope. Unlike eigenfunctions, which may be complex-valued and sign-changing, supereigenfunctions are nonnegative and encode growth bounds. The focus of this paper is on the supereigenfunctions of the Koopman operator as defined above, however we define the subeigenfunctions as below for completeness. 

\begin{definition}[Subeigenfunction]
A nonnegative observable $\varphi:\Omega\to\mathbb R_{\ge 0}$ is called a
subeigenfunction of $\mathcal K_f$ with rate $\lambda\in\mathbb R$ if
\[
\mathcal K_f\varphi \ge \lambda \varphi.
\]
\end{definition}

\begin{proposition}[Exponential envelope]
\label{prop:exp_envelope}
If $\varphi$ satisfies \eqref{eq:supereigen}, then
\begin{equation}
\varphi(s_t(x)) \le e^{\lambda t}\varphi(x),
\qquad \forall t \ge 0,\ \forall x \in \Omega,
\label{eq:envelope}
\end{equation}
equivalently $U_t \varphi \leq e^{\lambda t} \varphi$.
\end{proposition}

\begin{proof}
Apply Gr\"onwall's inequality.\qed
\end{proof}

The following connects equality-based and inequality-based representations.

\begin{proposition}[From eigenfunctions to envelopes]
\label{prop:eig_to_supereig}
If $\phi$ satisfies $\mathcal K_f \phi = \lambda \phi$ for some $\lambda \in \mathbb C$, then $\psi := |\phi|^2$ satisfies
\[
\mathcal K_f \psi = 2\,\mathrm{Re}(\lambda)\,\psi.
\]
Thus $\psi$ is both a supereigenfunction and a subeigenfunction with rate $2\,\mathrm{Re}(\lambda)$.
\end{proposition}

\begin{proof}
Since $
\psi=|\phi|^2=\phi\,\overline{\phi},
$
the product rule for the Koopman generator gives
\[
\mathcal K_f\psi
=
(\mathcal K_f\phi)\,\overline{\phi}
+
\phi\,\mathcal K_f\overline{\phi}.
\]
Because the vector field $f$ is real-valued, $
\mathcal K_f\overline{\phi}
=
\overline{\mathcal K_f\phi}.
$
Using the eigenfunction relation
$
\mathcal K_f\phi=\lambda\phi,
$
we therefore obtain
\[
\begin{aligned}
\mathcal K_f|\phi|^2
&=
\lambda\phi\overline{\phi}
+
\phi\,\overline{\lambda\phi}=
\lambda|\phi|^2+\overline{\lambda}|\phi|^2\\
&=
(\lambda+\overline{\lambda})|\phi|^2=
2\,\mathrm{Re}(\lambda)|\phi|^2.
\end{aligned}
\]
Hence, with $\psi=|\phi|^2$,
\[
\mathcal K_f\psi
=
2\,\mathrm{Re}(\lambda)\psi.
\]
Since $\psi\ge0$, equality implies both
\[
\mathcal K_f\psi
\le
2\,\mathrm{Re}(\lambda)\psi
\;\;{\rm 
and}
\;\;
\mathcal K_f\psi
\ge
2\,\mathrm{Re}(\lambda)\psi.
\]
Thus $\psi$ is simultaneously a supereigenfunction and a
subeigenfunction with rate $2\,\mathrm{Re}(\lambda)$.
\end{proof}

%=============================================================================
\subsection{Minimal rates and Collatz--Wielandt interpretation}
%=============================================================================

Supereigenvalues are not unique: if \eqref{eq:supereigen} holds for $\lambda$, it holds for any $\lambda' \ge \lambda$. This motivates an intrinsic rate.

\begin{definition}[Minimal supereigenvalue]
\label{def:min_rate}
For $\varphi \ge 0$, define
\begin{equation}
\lambda_\varphi := \inf\{\lambda \in \mathbb R : \mathcal K_f \varphi \le \lambda \varphi\}.
\label{eq:min_super}
\end{equation}
\end{definition}

If $\varphi \in C^1(\Omega)$ and $\varphi(x) > 0$, then
\begin{equation}
\lambda_\varphi = \sup_{x \in \Omega} \frac{(\mathcal K_f \varphi)(x)}{\varphi(x)}.
\label{eq:lambda_formula}
\end{equation}
This is a continuous-time analogue of the Collatz--Wielandt bound 
from Perron--Frobenius theory \cite{berman1994nonnegative,seneta2006nonnegative}, 
extended to positive operators and their generators 
\cite{nussbaum1986collatz,schaefer1974banach,engel2000one}. 
It provides the tightest exponential envelope certified by $\varphi$. 

%=============================================================================
\subsection{Cone properties and existence}
%=============================================================================

Supereigenfunctions form a convex cone i.e., \emph{Scaling:} $c\varphi$ is a supereigenfunction for $c>0$, \emph{Addition:} $\varphi_1 + \varphi_2$ has rate $\max\{\lambda_1,\lambda_2\}$. They also exist under mild conditions.

\begin{theorem}[Global supereigenfunction]
\label{thm:existence_basic}
Suppose
\[
\|f(x)\|\le L(1+\|x\|),
\qquad x\in\Omega.
\]
Then
\[
\varphi(x)=1+\|x\|^2
\]
is a global supereigenfunction satisfying
\[
\mathcal K_f\varphi(x)\le 3L\,\varphi(x),
\qquad x\in\Omega.
\]
\end{theorem}

\begin{proof}
Since $\nabla\varphi(x)=2x$,
\[
\mathcal K_f\varphi(x)
=
2x^\top f(x)
\le
2\|x\|\,\|f(x)\|
\le
2L\|x\|(1+\|x\|).
\]
Using $2\|x\|\le1+\|x\|^2$,
\[
\begin{aligned}
\mathcal K_f\varphi(x)
&\le
2L\|x\|+2L\|x\|^2\\
&\le
L(1+\|x\|^2)+2L\|x\|^2\\
&\le
3L(1+\|x\|^2)
=
3L\varphi(x).
\end{aligned}
\]
\end{proof}

\subsection{Examples: eigenfunctions versus supereigenfunctions}
\label{sec:examples}
%%%%%%%%%%%%%%%%%%%%%%%%%%%%%%%%%%%%%%%%%%%%%%%%%%%%%%%%%%%%%%%%%%%%%%%%%%%%%%

The following examples highlight the distinction between Koopman eigenfunctions,
which encode exact spectral evolution, and supereigenfunctions, which provide
nonnegative growth envelopes.

%=====================================================================
\subsubsection{Linear systems}
\label{subsec:ex_linear}
%=====================================================================

Consider the linear system \(\dot x = Ax\) on \(\mathbb R^n\).
If \(w_i^\top A = \mu_i w_i^\top\), then
$
\phi_i(x) = w_i^\top x$
is a Koopman eigenfunction satisfying $
\mathcal K_f \phi_i = \mu_i \phi_i$.
These eigenfunctions may be complex-valued. A nonnegative quadratic observable
\[
\psi_i(x) := |\phi_i(x)|^2,
\]
gives
$
\mathcal K_f \psi_i = 2\,\mathrm{Re}(\mu_i)\,\psi_i,
\;\;
\psi_i(s_t(x)) = e^{2\mathrm{Re}(\mu_i)t}\psi_i(x).
$
Thus \(\psi_i\) is a supereigenfunction with rate \(2\,\mathrm{Re}(\mu_i)\).
This shows that in the linear setting supereigenfunctions recover the 
quadratic growth envelopes associated with spectral modes.

%===========================================================
\subsubsection{Triangular nonlinear system: resonance and loss of smooth eigenfunctions}
\label{sec:triangular_example}
%===========================================================

Consider
\begin{equation}
\dot x_1 = \mu x_1,
\qquad
\dot x_2 = \beta x_2 + x_1^2,
\qquad (\mu,\beta\in\mathbb R).
\label{eq:triangular_system}
\end{equation}
The observable \(\phi_1(x)=x_1\) is always a Koopman eigenfunction with eigenvalue \(\mu\).
To construct an eigenfunction associated with \(\beta\), use the ansatz
\[
\phi_2(x)=x_2+h(x_1),
\]
which leads to
$
\mu x_1 h'(x_1)-\beta h(x_1)+x_1^2=0.
$
If \(\beta\neq 2\mu\), this admits the smooth solution
\[
\phi_2(x)=x_2-\frac{x_1^2}{2\mu-\beta},
\]
so \(\phi_2\) is a Koopman eigenfunction with eigenvalue \(\beta\).
In the resonant case \(\beta=2\mu\), however, the solution involves logarithmic terms and is not smooth at the origin; hence no smooth eigenfunction of this form exists. Despite this failure of spectral equality, one can still construct a global supereigenfunction. For example, for any \(k>0\),
\begin{equation}
\varphi(x):=x_1^2+x_2^2+kx_1^4
\label{eq:supereig_resonant}
\end{equation}
satisfies $
\mathcal K_f \varphi \le \lambda \varphi$
where $\lambda
=
\max\left\{0,
2\mu,\;
4\mu+\frac{1}{\sqrt{k}}
\right\}
$. Thus even when smooth eigenfunctions break down under resonance, a  supereigenfunction remains.

%=============================================================================
\subsubsection{Local eigenfunctions versus global supereigenfunctions}
\label{subsec:ex_bistable}
%=============================================================================

Consider the scalar nonlinear system
\begin{equation}
\dot x = -x + x^3 = -x(1-x^2),
\qquad x\in\mathbb R.
\label{eq:bistable}
\end{equation}
The origin is locally asymptotically stable with domain of attraction \(|x|<1\), while \(x=\pm1\) are unstable equilibria. For this system, one obtains the Koopman eigenfunction
\begin{equation}
\phi(x)=\frac{x}{\sqrt{1-x^2}},
\label{eq:local_eig}
\end{equation}
which satisfies
$
\mathcal K_f \phi = -\phi.
$
However, $\phi$ is well defined only for \(|x|<1\) and diverges as \(x\to \pm1\). 
Thus, the eigenfunction is confined to the basin of attraction of the origin and cannot be extended globally. In contrast, consider the globally defined nonnegative observable
\[
\varphi(x)=e^{-x^2}.
\]
A direct computation gives
\[
\mathcal K_f \varphi
=
\nabla \varphi(x) f(x)
=
(-2x e^{-x^2})(-x+x^3)
=
(2x^2-2x^4)\varphi(x).
\]
Since $
\sup_{x\in\mathbb R} (2x^2-2x^4) = \tfrac12,
$
it follows that
\[
\mathcal K_f \varphi \le \tfrac12\,\varphi(x),
\]
so $\varphi$ is a global supereigenfunction with rate $\lambda=\tfrac12$. This example illustrates a key advantage of the inequality-based formulation:
Koopman eigenfunctions may be only locally defined, whereas supereigenfunctions 
can be constructed globally, valid beyond 
the basin of attraction.
%=====================================================================
\subsubsection{Mechanical energy as a global supereigenfunction}
\label{subsec:ex_mechanical_energy}
%=====================================================================

Consider a mechanical system with generalized coordinates \(q\in\mathbb R^n\) and velocities \(v=\dot q\in\mathbb R^n\), governed by
\begin{equation}
M(q)\dot v + C(q,v)v + \nabla V(q) + D(q,v)v = 0,
\label{eq:mech_msd}
\end{equation}
together with \(\dot q = v\). Here \(M(q)\in\mathbb S_{++}^n\) is the inertia matrix, \(V(q)\) is the potential energy, and \(D(q,v)\in\mathbb S_+^n\) is a damping matrix. Assume the standard identity $\dot M(q)-2C(q,v)\ \text{is skew-symmetric}.$
Define the total energy
\begin{equation}
E(q,v)=\frac12 v^\top M(q)v + V(q).
\label{eq:total_energy}
\end{equation}
Then,
\begin{equation}
\frac{d}{dt}E(q(t),v(t))
=
-\,v(t)^\top D(q(t),v(t))v(t)
\le 0,
\label{eq:energy_dissipation}
\end{equation}
and therefore
\begin{equation}
\mathcal K_f E(x)\le 0.
\label{eq:energy_supereig}
\end{equation}
Hence \(E\) is a global supereigenfunction with admissible rate \(\lambda=0\). This example is representative of the certificate viewpoint developed in the paper: the total energy is not generally a Koopman eigenfunction, but it is a physically meaningful, globally defined inequality certificate directly aligned with stability and control analysis.
%%%%%%%%%%%%%%%%%%%%%%%%%%%%%%%%%%%%%%%%%%%%%%%%%%%%%%%%%%%%%%%%%%%%%%%%%%%%%%
\section{Directional Supereigenfunctions}
\label{sec:directional}
%%%%%%%%%%%%%%%%%%%%%%%%%%%%%%%%%%%%%%%%%%%%%%%%%%%%%%%%%%%%%%%%%%%%%%%%%%%%%%

Scalar supereigenfunctions provide global envelopes of the form
$
\mathcal K_f \varphi \le \lambda \varphi,
$
but a single rate may be conservative when the dynamics exhibit anisotropic growth or contraction. This motivates \emph{directional} supereigenfunctions, i.e., families
\[
\mathcal K_f \varphi_i \le \lambda_i \varphi_i, \qquad i=1,\dots,m,
\]
with each component capturing a distinct direction and rate. We present two complementary constructions: a geometric one based on tangent dynamics and the multiplicative ergodic theorem (MET), and one based on resolvent equations.

%=============================================================================
\subsection{Directional envelopes from tangent dynamics}
\label{subsec:met_directional}
%=============================================================================

Directional supereigenfunctions arise naturally from tangent dynamics, where 
anisotropic growth of perturbations induces multiple rates. We formalize this 
through metric-based constructions and their connection to the  MET.

%-----------------------------------------------------------------------------
\subsubsection{Tangent dynamics and metric characterization}
%-----------------------------------------------------------------------------

Let $f\in C^1(\Omega,\mathbb R^n)$ generate the flow $s_t$. The variational system
\begin{equation}
\dot v = Df(x(t))\,v,
\qquad v(t)=Y(t,x)v(0),
\label{eq:variational}
\end{equation}
defines the tangent cocycle $Y(t,x):=Ds_t(x)$. 
For observables $\varphi:\Omega\times\mathbb R^n\to\mathbb R$ defined on the tangent bundle, the extended dynamics
\begin{align}
\dot x=f(x),\qquad \dot v = Df(x)v \label{tangent_dynamics}
\end{align}
induces the tangent-space Koopman operator and generator 
\begin{align}(U^{\tan}_t \varphi)(x,v)&:=\varphi(s_t(x),Y(t,x)v)\nonumber\\
(\mathcal K_f^{\tan}\varphi)(x,v)
&:=
\nabla_x \varphi(x,v)^\top f(x)
+
\nabla_v \varphi(x,v)^\top Df(x)v.
\end{align} 
% For a matrix field $M:\Omega\to\mathbb S_+^n$, define
% \begin{equation}
% \varphi_M(x,v):=v^\top M(x)v.
% \label{eq:phi_metric}
% \end{equation}
% Then
% \begin{equation}
% \frac{d}{dt}\varphi_M
% =
% v^\top\!\big(Df^\top M+MDf+\dot M\big)v,
% \label{eq:metric_derivative}
% \end{equation}
% where $\dot M=(\nabla M)\cdot f$.

\begin{proposition}[Metric characterization]
\label{prop:metric_char}
Let
$
\varphi_M(x,v):=v^\top M(x)v,
\;\; M:\Omega\to\mathbb S_+^n.
$
Then
\[
(\mathcal K_f^{\tan}\varphi_M)(x,v)\le \lambda\,\varphi_M(x,v)
\qquad \forall (x,v)\in\Omega\times\mathbb R^n
\]
if and only if
\[
Df(x)^\top M(x)+M(x)Df(x)+\dot M(x)\preceq \lambda\,M(x),
\;\; \forall x\in\Omega,
\]
\[
\dot M(x):=(\nabla M(x))\cdot f(x)=
\sum_{k=1}^n
\frac{\partial M(x)}{\partial x_k}f_k(x).
\]
\end{proposition}
\begin{proof}
The result follows by differentiating
\[
\varphi_M(x,v)=v^\top M(x)v
\]
along the tangent dynamics (\ref{tangent_dynamics})
which gives
\[
(\mathcal K_f^{\tan}\varphi_M)(x,v)
=
v^\top
\big(
Df^\top M+MDf+\dot M
\big)v.
\]
The inequality then holds for all \(v\) if and only if the stated matrix inequality is satisfied.\qed
\end{proof}
Thus, directional supereigenfunctions correspond to matrix differential 
inequalities, closely related to contraction metrics.

%-----------------------------------------------------------------------------
\subsubsection{Asymptotic construction via MET}
%-----------------------------------------------------------------------------

Assume $\mu$ is invariant and $\log^+\|Df\|\in L^1(\mu)$. By the MET, for $\mu$-a.e.\ $x$ there exist Lyapunov exponents
$\chi_1>\cdots>\chi_r$ and an Oseledets splitting
\[
T_x\Omega = E_1(x)\oplus\cdots\oplus E_r(x),\;\; \lim_{t\to\infty}\frac{1}{t}\log\|Y(t,x)v\|=\chi_i.
\]
for $v\in E_i(x)\setminus\{0\}$.
% \[
% \lim_{t\to\infty}\frac{1}{t}\log\|Y(t,x)v\|=\chi_i.
% \]
% By  MET \cite{oseledets1968multiplicative,ruelle1979ergodic}, for $\mu$-a.e.\ $x$,
% \[
% T_x\Omega = \bigoplus_{i=1}^r E_i(x), 
% \qquad 
% \|Y(t,x)v\|\sim e^{\chi_i t},\ \ v\in E_i(x).
% \]
Let $P_i(x)$ denote the orthogonal projector onto $E_i(x)$ and define
\begin{equation}
\varphi_i(x,v):=\|P_i(x)v\|^2.
\label{eq:met_phi_iv}
\end{equation}

\begin{theorem}[Directional supereigenfunctions from MET]
\label{thm:met_supereig_asympt}
For $\mu$-a.e.\ $x$ and $P_i(x)v\neq 0$, for $\varphi_i(x,v)$ as defined in (\ref{eq:met_phi_iv}), we have
\[
\lim_{t\to\infty}\frac{1}{t}\log 
\varphi_i(s_t(x),Y(t,x)v)=2\chi_i.
\]
Equivalently, for any $\varepsilon>0$,
\[
\varphi_i(s_t(x),Y(t,x)v)
\le e^{(2\chi_i+\varepsilon)t}\varphi_i(x,v)
\quad\text{for large }t.
\]
\end{theorem}

\begin{proof}
By the multiplicative ergodic theorem (MET) \cite{oseledets1968multiplicative}, for $\mu$-a.e.\ $x$ there exists an invariant splitting
\[
T_x\Omega = E_1(x)\oplus\cdots\oplus E_r(x),\;\;
Y(t,x)E_j(x)=E_j(s_t(x)),
\]
with associated Lyapunov exponents $\chi_j$. Let $P_i(x)$ denote the projector onto $E_i(x)$. Then
\[
P_i(s_t(x))\,Y(t,x)=Y(t,x)\,P_i(x).
\]
Hence, for any $v$ such that $P_i(x)v\neq 0$,
\begin{align}
&\varphi_i\bigl(s_t(x),Y(t,x)v\bigr)
=
\bigl\|P_i(s_t(x))Y(t,x)v\bigr\|^2\nonumber\\
&=
\bigl\|Y(t,x)P_i(x)v\bigr\|^2.
\end{align}
Since $P_i(x)v \in E_i(x)\setminus\{0\}$, the MET gives
\[
\lim_{t\to\infty}\frac{1}{t}\log \bigl\|Y(t,x)P_i(x)v\bigr\|=\chi_i.
\]
Therefore,
\[
\lim_{t\to\infty}\frac{1}{t}\log 
\varphi_i\bigl(s_t(x),Y(t,x)v\bigr)=2\chi_i.
\]
Consequently, for any $\varepsilon>0$, there exists 
$T=T(x,v,\varepsilon)>0$ such that for all $t\ge T$,
\[
(U^{\tan}_t\varphi_i)(x,v)
\le
e^{(2\chi_i+\varepsilon)t}\varphi_i(x,v).
\]
Thus, $\varphi_i$ defines an asymptotic directional supereigenfunction
with rate $2\chi_i$ along typical trajectories.\qed
\end{proof}
% Thus, MET yields intrinsic directional supereigenfunctions with rates $2\chi_i$.
%-----------------------------------------------------------------------------
\subsubsection{Adapted metrics from Oseledets directions}
%-----------------------------------------------------------------------------

For weights $w_i>0$, define
\begin{equation}
M(x):=\sum_{i=1}^r w_i P_i(x),
\;\;
\varphi_M(x,v)=\sum_{i=1}^r w_i\|P_i(x)v\|^2.
\label{eq:adapted_metric}
\end{equation}

\begin{proposition}[Adapted metric supereigenfunction]
\label{prop:adapted_metric}
$\varphi_M$ is an asymptotic supereigenfunction with rate 
$2\chi_{\max}$, where $\chi_{\max}=\max\{\chi_i:\ w_i>0\}$.
\end{proposition}

This provides a systematic construction of contraction-type metrics from 
the MET decomposition.

%-----------------------------------------------------------------------------
\subsubsection{Infinite-horizon Gramian construction}
\label{subsec:gramian_directional}
%=============================================================================

A constructive realization of directional supereigenfunctions can be obtained 
directly from tangent dynamics via an infinite-horizon Gramian.

Let $Q:\Omega\to\mathbb S_+^n$ and $\lambda>0$. Define
\begin{equation}
M_{\lambda}(x)
:=
\int_0^\infty e^{-2\lambda t}\,
Y(t,x)^\top Q\big(s_t(x)\big)\,Y(t,x)\,dt,
\label{eq:gramian_metric}
\end{equation}
whenever the integral converges. The associated observable is
\begin{equation}
\varphi_{\lambda}(x,v)
=
v^\top M_{\lambda}(x)v
=
\int_0^\infty e^{-2\lambda t}
\|Q^{1/2}(s_t(x))Y(t,x)v\|^2 dt.
\label{eq:gramian_phi}
\end{equation}
\begin{proposition}[Family of Gramian-based directional supereigenfunctions]
\label{prop:gramian_family}
Let $Q_i:\Omega\to\mathbb S_+^n$ and $\lambda_i>0$, $i=1,\dots,m$, and define
\begin{equation}
M_i(x)
:=
\int_0^\infty e^{-2\lambda_i t}\,
Y(t,x)^\top Q_i\big(s_t(x)\big)\,Y(t,x)\,dt,
\label{eq:gramian_metric_family}
\end{equation}
whenever the integral converges. For each $i$, define
\begin{equation}
\varphi_i(x,v)
:=
v^\top M_i(x)v
=
\int_0^\infty e^{-2\lambda_i t}
\|Q_i^{1/2}(s_t(x))\,Y(t,x)v\|^2\,dt.
\label{eq:gramian_phi_family}
\end{equation}
Then each $\varphi_i$ satisfies
\begin{equation}
(\mathcal K_f^{\mathrm{tan}}\varphi_i)(x,v)
=
2\lambda_i\,\varphi_i(x,v)-v^\top Q_i(x)v,
\label{eq:gramian_resolvent_family}
\end{equation}
and hence
\begin{equation}
(\mathcal K^{\mathrm{tan}}_f\varphi_i)(x,v)\le 2\lambda_i\,\varphi_i(x,v).
\label{eq:gramian_supereig_family}
\end{equation}
Therefore, $\{\varphi_i\}_{i=1}^m$ defines a family of directional supereigenfunctions on the tangent bundle.
\end{proposition}
\begin{remark}[Relation to Lyapunov exponents]
The convergence of \eqref{eq:gramian_metric} requires the discount rate 
$\lambda$ to dominate the asymptotic growth of the tangent dynamics. In 
particular, if $\chi_{\max}$ denotes the largest Lyapunov exponent in the 
directions weighted by $Q$, then a sufficient condition for convergence is 
$\lambda>\chi_{\max}$. Thus, $\lambda$ acts as a design upper bound on 
directional growth rates, linking the Gramian construction to the MET-based 
characterization.
\end{remark}

\subsection{Data-to-Certificate Construction via Positive Resolvents}
\label{sec:resolvent_construction}
%===========================================================

We next develop a trajectory-based construction of Koopman
supereigenfunctions using the resolvent of the Koopman generator.
The central idea is to select a nonnegative probe function that
encodes the quantity of interest and map it directly to a certificate
through the positive resolvent. This provides a data-to-certificate
(D2C) construction: the certificate can be evaluated from trajectory
rollouts without first identifying an intermediate dynamical model or
a finite-dimensional Koopman representation.

Let $\{U_t\}_{t\ge0}$ denote the positive Koopman $C_0$-semigroup
on the Banach lattice $\mathcal F$, with infinitesimal generator
$\mathcal K_f$. Assume that the semigroup is exponentially bounded,
\begin{equation}
\|U_t\|\le M e^{\omega t},
\qquad t\ge0,
\label{eq:koopman_semigroup_growth}
\end{equation}
for some $M\ge1$ and $\omega\in\mathbb R$. For
$\lambda>\omega$, the resolvent of $\mathcal K_f$ is
\begin{equation}
R(\lambda;\mathcal K_f)
:=
(\lambda I-\mathcal K_f)^{-1}
=
\int_0^\infty e^{-\lambda t}U_t\,dt .
\label{eq:koopman_resolvent}
\end{equation}
Since $U_t$ is positive, $R(\lambda;\mathcal K_f)$ is also a
positive operator. Thus, a nonnegative probe $g\in\mathcal F_+$
is mapped to a nonnegative observable
\begin{equation}
\varphi_\lambda
=
R(\lambda;\mathcal K_f)g
=
\int_0^\infty e^{-\lambda t}U_tg\,dt
=
\int_0^\infty e^{-\lambda t}g(s_t(x))\,dt .
\label{eq:resolvent_supereig}
\end{equation}
The following theorem summarizes the resulting exact,
finite-rollout, and approximate certificate properties.

\begin{theorem}[Positive-resolvent data-to-certificate theorem]
\label{thm:d2c_resolvent}
Let $\{U_t\}_{t\ge0}$ be a positive Koopman $C_0$-semigroup on
$\mathcal F$ satisfying \eqref{eq:koopman_semigroup_growth}, and let
$g\in\mathcal F_+$. For $\lambda>\omega$, define
\begin{equation}
\varphi_\lambda
:=
R(\lambda;\mathcal K_f)g.
\label{eq:d2c_exact_resolvent}
\end{equation}
Then the following statements hold.

\begin{enumerate}

\item[(i)] \textbf{Exact positive certificate.}
The function $\varphi_\lambda\in D(\mathcal K_f)$ is nonnegative and
satisfies
\begin{equation}
(\lambda I-\mathcal K_f)\varphi_\lambda=g.
\label{eq:d2c_resolvent_identity}
\end{equation}
Consequently,
\begin{equation}
\mathcal K_f\varphi_\lambda
=
\lambda\varphi_\lambda-g
\le
\lambda\varphi_\lambda,
\label{eq:d2c_exact_supereig}
\end{equation}
and hence $\varphi_\lambda$ is a Koopman supereigenfunction with
rate $\lambda$.

\item[(ii)] \textbf{Finite-rollout certificate.}
For $T>0$, define
\begin{equation}
\varphi_{\lambda,T}
:=
\int_0^T e^{-\lambda t}U_tg\,dt .
\label{eq:d2c_finite_resolvent}
\end{equation}
Then
\begin{equation}
\mathcal K_f\varphi_{\lambda,T}
=
\lambda\varphi_{\lambda,T}
-g
+
e^{-\lambda T}U_Tg.
\label{eq:d2c_finite_identity}
\end{equation}
Moreover,
\begin{equation}
\|\varphi_\lambda-\varphi_{\lambda,T}\|
\le
\frac{M}{\lambda-\omega}
e^{-(\lambda-\omega)T}\|g\|,
\label{eq:d2c_truncation_error}
\end{equation}
and therefore
\begin{equation}
\mathcal K_f\varphi_{\lambda,T}
\le
\lambda\varphi_{\lambda,T}
+
\varepsilon_T,
\qquad
\varepsilon_T
:=
M e^{-(\lambda-\omega)T}\|g\|.
\label{eq:d2c_finite_supereig}
\end{equation}

\item[(iii)] \textbf{Approximate data-driven certificate.}
Let $\widehat{\varphi}_{\lambda,T}\in D(\mathcal K_f)$ be an
approximation of $\varphi_{\lambda,T}$ satisfying
\begin{equation}
\|\widehat{\varphi}_{\lambda,T}
-\varphi_{\lambda,T}\|
\le \varepsilon_0,
\qquad
\|\mathcal K_f(
\widehat{\varphi}_{\lambda,T}
-\varphi_{\lambda,T})\|
\le \varepsilon_1 .
\label{eq:d2c_graph_error}
\end{equation}
Then
\begin{equation}
\mathcal K_f\widehat{\varphi}_{\lambda,T}
\le
\lambda\widehat{\varphi}_{\lambda,T}
+
\varepsilon_{\mathrm{D2C}},
\label{eq:d2c_approx_supereig}
\end{equation}
where
\begin{equation}
\varepsilon_{\mathrm{D2C}}
:=
M e^{-(\lambda-\omega)T}\|g\|
+
\varepsilon_1
+
|\lambda|\varepsilon_0 .
\label{eq:d2c_total_error}
\end{equation}
Hence, as $T\to\infty$ and
$\varepsilon_0,\varepsilon_1\to0$, the approximate certificate
converges, in the graph-norm sense specified above, to an exact
Koopman supereigenfunction.
\end{enumerate}
\end{theorem}

\begin{proof}
Since $U_t$ is positive and $g\in\mathcal F_+$,
$U_tg\ge0$ for all $t\ge0$. Hence
$\varphi_\lambda=R(\lambda;\mathcal K_f)g\ge0$.
For $\lambda>\omega$, the Laplace representation of the resolvent
gives
\[
(\lambda I-\mathcal K_f)
R(\lambda;\mathcal K_f)g=g,
\]
which proves \eqref{eq:d2c_resolvent_identity} and therefore
\eqref{eq:d2c_exact_supereig}.

For the finite-horizon construction, differentiating the semigroup
representation gives
\[
\mathcal K_f\varphi_{\lambda,T}
=
\lambda\varphi_{\lambda,T}
-g+e^{-\lambda T}U_Tg.
\]
Furthermore,
\begin{align}
\|\varphi_\lambda-\varphi_{\lambda,T}\|
&\le
\int_T^\infty e^{-\lambda t}\|U_tg\|\,dt \nonumber\\
&\le
M\|g\|
\int_T^\infty e^{-(\lambda-\omega)t}\,dt,
\end{align}
which yields \eqref{eq:d2c_truncation_error}. The growth bound
\[
\|e^{-\lambda T}U_Tg\|
\le
M e^{-(\lambda-\omega)T}\|g\|
\]
then gives \eqref{eq:d2c_finite_supereig}.

Finally, let
$
e=\widehat{\varphi}_{\lambda,T}-\varphi_{\lambda,T}.
$
Using \eqref{eq:d2c_finite_identity},
\[
\mathcal K_f\widehat{\varphi}_{\lambda,T}
=
\lambda\widehat{\varphi}_{\lambda,T}
-g
+
e^{-\lambda T}U_Tg
+
\mathcal K_fe-\lambda e.
\]
Since $g\ge0$, the semigroup growth bound together with
\eqref{eq:d2c_graph_error} gives
\[
\mathcal K_f\widehat{\varphi}_{\lambda,T}
\le
\lambda\widehat{\varphi}_{\lambda,T}
+
M e^{-(\lambda-\omega)T}\|g\|
+
\varepsilon_1
+
|\lambda|\varepsilon_0,
\]
which proves \eqref{eq:d2c_approx_supereig}.
\end{proof}

\begin{remark}[Semigroup versus probe-dependent growth rates]
\label{rem:resolvent_growth_rates}
The condition $\lambda>\omega$ in
Theorem~\ref{thm:d2c_resolvent} is a function-space condition
ensuring existence of the resolvent $R(\lambda;\mathcal K_f)$.
For a particular probe $g$, however, the trajectory representation
\eqref{eq:resolvent_supereig} may converge under a weaker,
probe-dependent condition.

In particular, define
\begin{equation}
\lambda_g^\star
:=
\limsup_{t\to\infty}
\frac{1}{t}
\log
\sup_{x\in\Omega}
g(s_t(x)).
\label{eq:probe_growth_rate}
\end{equation}
Whenever this quantity is finite, any
$\lambda>\lambda_g^\star$ yields convergence of the
probe-specific trajectory integral under the corresponding
exponential growth bound. Thus, $\omega$ provides a uniform
function-space condition, whereas $\lambda_g^\star$ provides a
potentially less conservative condition tailored to the selected
probe.
\end{remark}

For data-driven implementation, the probe-dependent growth rate
can be estimated directly from trajectory rollouts. Given sampled
initial conditions $\{x^{(j)}\}_{j=1}^{N}$ and a rollout horizon
$T$, define
\begin{equation}
\widehat{\lambda}_g(T)
:=
\frac{1}{T}
\log
\left(
\max_{1\le j\le N}
\frac{g(s_T(x^{(j)}))+\epsilon}
     {g(x^{(j)})+\epsilon}
\right),
\qquad \epsilon>0.
\label{eq:empirical_probe_growth}
\end{equation}
A practical choice is then
\begin{equation}
\lambda
=
\widehat{\lambda}_g(T)+\delta,
\qquad
\delta>0,
\label{eq:lambda_selection}
\end{equation}
where $\delta$ provides a margin for finite-horizon and sampling
effects.

The choice of probe $g$ determines what dynamical quantity is
encoded by the resulting certificate. For example, probes that
emphasize individual state coordinates or selected state
directions produce corresponding growth envelopes, whereas probes
that penalize unsafe regions produce risk-sensitive certificates.
Thus, different nonnegative probes generate a family of
task-dependent supereigenfunctions through the same positive
resolvent construction.

For a collection of probes
$g_i\in\mathcal F_+$, $i=1,\ldots,m$, the construction can be
applied componentwise:
\begin{equation}
\varphi_i
=
R(\lambda_i;\mathcal K_f)g_i,
\qquad i=1,\ldots,m.
\label{eq:vector_resolvent_construction}
\end{equation}
Defining
\[
\Phi(x)
=
\begin{bmatrix}
\varphi_1(x) & \cdots & \varphi_m(x)
\end{bmatrix}^{\!\top},
\qquad
\Lambda
=
\operatorname{diag}(\lambda_1,\ldots,\lambda_m),
\]
gives the componentwise inequality
\begin{equation}
\mathcal K_f\Phi(x)
\preceq
\Lambda\Phi(x),
\label{eq:vector_resolvent_supereig}
\end{equation}
which provides the starting point for the comparison bounds and
control constructions developed in the next section.

\section{Certificate-Based Analysis and Control}
\label{sec:application_modeling}
%===========================================================

The constructions in the previous section produce nonnegative
observables whose evolution is bounded by linear comparison
dynamics. In particular, given a collection of supereigenfunctions
$\{\varphi_i\}_{i=1}^m$, define the certificate coordinate map
\begin{equation}
\Phi(x)
:=
\begin{bmatrix}
\varphi_1(x)\\
\vdots\\
\varphi_m(x)
\end{bmatrix}
\in\mathbb R_{\ge0}^m,
\qquad
\Lambda
:=
\operatorname{diag}(\lambda_1,\ldots,\lambda_m).
\label{eq:certificate_coordinates}
\end{equation}
The componentwise inequalities
\begin{equation}
\mathcal K_f\Phi(x)
\preceq
\Lambda\Phi(x)
\label{eq:certificate_generator_ineq}
\end{equation}
imply that the evolution of $\Phi(x(t))$ is dominated by a
positive linear comparison system.

This observation provides the common mechanism underlying the
applications developed below. Rather than requiring an exact
finite-dimensional model of the state dynamics, the certificate
coordinates provide upper envelopes on selected dynamical
quantities. These envelopes yield set-valued uncertainty
predictions, stability guarantees when the comparison dynamics
are contractive, and affine constraints for control synthesis.

\begin{theorem}[Comparison principle in certificate coordinates]
\label{thm:certificate_comparison}
Let $\Phi:\Omega\rightarrow\mathbb R_{\ge0}^m$ be continuously
differentiable and suppose
\begin{equation}
\mathcal K_f\Phi(x)
\preceq
\Lambda\Phi(x),
\qquad x\in\Omega,
\label{eq:exact_certificate_ineq}
\end{equation}
where
$\Lambda=\operatorname{diag}(\lambda_1,\ldots,\lambda_m)$.
Then, along every trajectory $x(t)=s_t(x_0)$,
\begin{equation}
\Phi(x(t))
\preceq
e^{\Lambda t}\Phi(x_0),
\qquad t\ge0.
\label{eq:exact_certificate_comparison}
\end{equation}
Equivalently,
\begin{equation}
\varphi_i(x(t))
\le
e^{\lambda_i t}\varphi_i(x_0),
\qquad
i=1,\ldots,m.
\label{eq:exact_scalar_comparison}
\end{equation}
\end{theorem}

\begin{proof}
Along a trajectory $x(t)=s_t(x_0)$,
\[
\frac{d}{dt}\varphi_i(x(t))
=
(\mathcal K_f\varphi_i)(x(t))
\le
\lambda_i\varphi_i(x(t)).
\]
Gronwall's inequality gives
$\varphi_i(x(t))
\le e^{\lambda_i t}\varphi_i(x_0)$.
Stacking the componentwise inequalities gives
\eqref{eq:exact_certificate_comparison}.\qed
\end{proof}

\begin{corollary}[Comparison with approximate D2C certificates]
\label{cor:approx_certificate_comparison}
Let
$\widehat{\Phi}:\Omega\rightarrow\mathbb R_{\ge0}^m$
satisfy
\begin{equation}
\mathcal K_f\widehat{\Phi}(x)
\preceq
\Lambda\widehat{\Phi}(x)+\varepsilon,
\label{eq:approx_certificate_ineq}
\end{equation}
where
$\varepsilon=[\varepsilon_1,\ldots,\varepsilon_m]^\top
\in\mathbb R_{\ge0}^m$.
Then
\begin{equation}
\widehat{\Phi}(x(t))
\preceq
e^{\Lambda t}\widehat{\Phi}(x_0)
+
\int_0^t
e^{\Lambda(t-\tau)}\varepsilon\,d\tau .
\label{eq:approx_certificate_comparison}
\end{equation}
For $\lambda_i\neq0$, this gives
\begin{equation}
\widehat{\varphi}_i(x(t))
\le
e^{\lambda_i t}\widehat{\varphi}_i(x_0)
+
\frac{\varepsilon_i}{\lambda_i}
\left(e^{\lambda_i t}-1\right),
\label{eq:approx_scalar_comparison}
\end{equation}
whereas for $\lambda_i=0$,
\begin{equation}
\widehat{\varphi}_i(x(t))
\le
\widehat{\varphi}_i(x_0)+\varepsilon_i t.
\label{eq:approx_scalar_zero}
\end{equation}
\end{corollary}

\begin{proof}
For each component,
\[
\frac{d}{dt}\widehat{\varphi}_i(x(t))
\le
\lambda_i\widehat{\varphi}_i(x(t))+\varepsilon_i.
\]
The result follows from the scalar comparison lemma, applied
componentwise.\qed
\end{proof}

For the finite-rollout and data-driven certificates constructed in
Theorem~\ref{thm:d2c_resolvent}, the residual vector
$\varepsilon$ in
Corollary~\ref{cor:approx_certificate_comparison}
is determined explicitly by the certificate approximation error.
In particular, for the $i$th data-driven certificate one may take
\[
\varepsilon_i
=
\varepsilon_{\mathrm{D2C},i}.
\]
Hence the effects of finite rollout horizon and certificate
approximation propagate directly into the comparison envelope as
an additive, explicitly quantified uncertainty term.

%=============================================================================
\subsection{Uncertainty Propagation via Inequality Lifts}
\label{subsec:modeling}
%=============================================================================

The comparison bounds above provide a direct mechanism for
propagating uncertainty. Rather than predicting individual
trajectories, the certificate coordinates define certified outer
envelopes for the evolution of selected observables. For exact
supereigenfunctions these envelopes follow from
Theorem~\ref{thm:certificate_comparison}; for finite-rollout or
data-driven certificates they follow from
Corollary~\ref{cor:approx_certificate_comparison}, with the D2C
residual appearing explicitly as an inflation of the envelope.

\begin{corollary}[Forward uncertainty bounds]
\label{cor:uncertainty}
Let $X_0\subset\Omega$ be a set of initial conditions and define
\[
\bar\varphi_i(0):=\sup_{x_0\in X_0}\widehat\varphi_i(x_0).
\]
Suppose
\[
\mathcal K_f\widehat\Phi
\preceq
\Lambda\widehat\Phi+\varepsilon,
\qquad
\varepsilon\in\mathbb R_{\ge0}^m.
\]
Then the reachable set
\[
\mathcal R_t:=\{s_t(x_0):x_0\in X_0\}
\]
satisfies
\[
\mathcal R_t
\subseteq
\bigcap_{i=1}^m
\left\{
x\in\Omega:
\widehat\varphi_i(x)
\le
e^{\lambda_i t}\bar\varphi_i(0)
+
\eta_i(t)
\right\},
\]
where
\[
\eta_i(t):=
\begin{cases}
\dfrac{\varepsilon_i}{\lambda_i}
\left(e^{\lambda_i t}-1\right), & \lambda_i\neq0,\\[2mm]
\varepsilon_i t, & \lambda_i=0.
\end{cases}
\]
For exact supereigenfunctions, $\varepsilon=0$ and hence
$\eta_i(t)=0$.
\end{corollary}

\begin{proof}
For any $x_0\in X_0$, Corollary~\ref{cor:approx_certificate_comparison}
gives
\[
\widehat\varphi_i(s_t(x_0))
\le
e^{\lambda_i t}\widehat\varphi_i(x_0)+\eta_i(t)
\le
e^{\lambda_i t}\bar\varphi_i(0)+\eta_i(t).
\]
Taking all $i=1,\ldots,m$ yields the stated outer approximation.
\end{proof}
%=============================================================================
%=============================================================================
\subsection{Stability and Stabilizing Control via Supereigenfunctions}
\label{subsec:stability_control}
%=============================================================================

The comparison principle provides a direct route to stability
analysis and control design. Stability follows from contracting
certificate coordinates together with a coercivity condition relating
the certificates to the state, while control synthesis amounts to
shaping their open-loop growth rates to desired negative values.

\subsubsection{Stability from contracting certificates}

\begin{theorem}[Stability from contracting certificates]
\label{thm:stability}
Let $\varphi_i:\Omega\to\mathbb R_{\ge0}$, $i=1,\ldots,m$, satisfy
\[
\mathcal K_f\varphi_i(x)\le\lambda_i\varphi_i(x),
\qquad \lambda_i<0,
\]
and define
\[
\Phi(x):=(\varphi_1(x),\ldots,\varphi_m(x))^\top.
\]
Suppose there exist class-$\mathcal K_\infty$ functions
$\alpha_1,\alpha_2$ such that
\[
\alpha_1(\|x\|)
\le
\mathbf 1^\top\Phi(x)
\le
\alpha_2(\|x\|),
\qquad x\in\Omega.
\]
Then the origin is globally asymptotically stable, with
\[
\|x(t)\|
\le
\alpha_1^{-1}
\left(
e^{\lambda_{\max}t}\alpha_2(\|x(0)\|)
\right),
\qquad
\lambda_{\max}:=\max_i\lambda_i<0.
\]
\end{theorem}

\begin{proof}
By Theorem~\ref{thm:certificate_comparison},
\[
\mathbf 1^\top\Phi(x(t))
\le
e^{\lambda_{\max}t}\mathbf 1^\top\Phi(x(0)).
\]
The result follows directly from the coercivity bounds.\qed
\end{proof}

\begin{remark}
If, more specifically,
\[
c_1\|x\|^p
\le
\mathbf 1^\top\Phi(x)
\le
c_2\|x\|^p
\]
for some $c_1,c_2,p>0$, then
\[
\|x(t)\|
\le
\left(\frac{c_2}{c_1}\right)^{1/p}
e^{\lambda_{\max}t/p}\|x(0)\|,
\]
and the origin is globally exponentially stable.
\end{remark}

\subsubsection{Stabilizing control via inequality shaping}

Consider the control-affine system
\begin{equation}
\dot x=f(x)+G(x)u,
\qquad
u\in\mathcal U\subset\mathbb R^p.
\label{eq:control_affine}
\end{equation}
Let $\varphi_i$ be supereigenfunctions of the drift dynamics,
\begin{equation}
\mathcal K_f\varphi_i
\le
\lambda_i\varphi_i,
\qquad i=1,\ldots,m.
\label{eq:drift_supereig}
\end{equation}
For desired rates $\beta_i<0$, stabilization is achieved by enforcing
\begin{equation}
\mathcal K_f\varphi_i
+
(\mathcal K_G\varphi_i)u
\le
\beta_i\varphi_i,
\qquad i=1,\ldots,m,
\label{eq:closed_loop_rate}
\end{equation}
where
\[
\mathcal K_G\varphi_i:=\nabla\varphi_i^\top G.
\]

\begin{proposition}[Affine sufficient control constraints]
\label{prop:control_constraint}
A sufficient condition for \eqref{eq:closed_loop_rate} is
\begin{equation}
(\mathcal K_G\varphi_i)(x)u
\le
-(\lambda_i-\beta_i)\varphi_i(x),
\qquad i=1,\ldots,m.
\label{eq:affine_control_constraint}
\end{equation}
\end{proposition}

\begin{proof}
The result follows immediately by combining
$\mathcal K_f\varphi_i\le\lambda_i\varphi_i$
with \eqref{eq:affine_control_constraint}.
\end{proof}

Define
\[
B(x):=
\begin{bmatrix}
(\mathcal K_G\varphi_1)(x)\\
\vdots\\
(\mathcal K_G\varphi_m)(x)
\end{bmatrix},
\qquad
D:=
\operatorname{diag}
(\lambda_1-\beta_1,\ldots,\lambda_m-\beta_m).
\]
Then \eqref{eq:affine_control_constraint} becomes
\begin{equation}
B(x)u\preceq-D\Phi(x).
\label{eq:stacked_control_constraint}
\end{equation}

\begin{remark}[Approximate D2C certificates]
\label{rem:d2c_control}
For an approximate certificate satisfying
\[
\mathcal K_f\widehat\varphi_i
\le
\lambda_i\widehat\varphi_i+\varepsilon_i,
\]
the sufficient control constraint becomes
\begin{equation}
(\mathcal K_G\widehat\varphi_i)u
\le
-(\lambda_i-\beta_i)\widehat\varphi_i-\varepsilon_i.
\label{eq:robust_d2c_control}
\end{equation}
Thus, the D2C residual enters the controller as an explicit
tightening of the affine certificate constraint. For the construction
of Theorem~\ref{thm:d2c_resolvent}, one may take
$\varepsilon_i=\varepsilon_{\mathrm{D2C},i}$.
\end{remark}

\begin{theorem}[Convex stabilizing controller]
\label{thm:stabilizing_qp}
Consider \eqref{eq:control_affine} and suppose:

\begin{enumerate}
\item[(i)] the certificates satisfy \eqref{eq:drift_supereig};

\item[(ii)] there exist $\beta_i<0$ such that
\eqref{eq:affine_control_constraint} is feasible for all
$x\in\Omega$;

\item[(iii)] there exist class-$\mathcal K_\infty$ functions
$\alpha_1,\alpha_2$ satisfying
\[
\alpha_1(\|x\|)
\le
\mathbf 1^\top\Phi(x)
\le
\alpha_2(\|x\|);
\]

\item[(iv)] the resulting closed-loop trajectories remain in $\Omega$.
\end{enumerate}

Let
\begin{equation}
u^\star(x)
\in
\arg\min_{u\in\mathcal U}
\frac12u^\top Ru
\quad
\text{s.t.}\quad
B(x)u\preceq-D\Phi(x),
\qquad R\succ0.
\label{eq:stabilizing_qp}
\end{equation}
Then the origin is globally asymptotically stable.
\end{theorem}

\begin{proof}
The QP constraints imply
\[
\mathcal K_f\varphi_i+
(\mathcal K_G\varphi_i)u^\star
\le
\beta_i\varphi_i.
\]
Hence
\[
\mathbf 1^\top\Phi(x(t))
\le
e^{\beta_{\max}t}\mathbf 1^\top\Phi(x(0)),
\qquad
\beta_{\max}:=\max_i\beta_i<0.
\]
The result follows from the coercivity bounds as in
Theorem~\ref{thm:stability}.
\end{proof}

\begin{theorem}[Geometric feasibility of certificate rate shaping]
\label{thm:vector_qp_feasibility_geometric}
Consider the control-affine system \eqref{eq:control_affine} and let
$\varphi_i:\Omega\to\mathbb R_{\ge0}$, $i=1,\ldots,m$, satisfy
\[
\mathcal K_f\varphi_i(x)\le\lambda_i\varphi_i(x).
\]
For desired rates $\beta_i<\lambda_i$, define
\[
a_i(x):=(\mathcal K_G\varphi_i)(x)\in\mathbb R^{1\times p},
\qquad
d_i(x):=(\lambda_i-\beta_i)\varphi_i(x)\ge0,
\]
and
\[
B(x):=
\begin{bmatrix}
a_1(x)\\
\vdots\\
a_m(x)
\end{bmatrix},
\qquad
d(x):=
\begin{bmatrix}
d_1(x)\\
\vdots\\
d_m(x)
\end{bmatrix}.
\]
Then the sufficient rate-shaping constraints
\begin{equation}
B(x)u\preceq-d(x)
\label{eq:rate_shaping_feasibility}
\end{equation}
are feasible at $x$ if and only if
\begin{equation}
\min_{u\in\mathcal U}
\max_{1\le i\le m}
\left\{a_i(x)u+d_i(x)\right\}
\le0.
\label{eq:minmax_feasibility}
\end{equation}

If
\[
\mathcal U=\{u\in\mathbb R^p:\|u\|_2\le u_{\max}\},
\]
a sufficient geometric condition for feasibility is the existence
of a unit vector $w(x)\in\mathbb R^p$ such that
\begin{equation}
a_i(x)w(x)
\le
-\frac{d_i(x)}{u_{\max}},
\qquad i=1,\ldots,m.
\label{eq:common_descent_direction}
\end{equation}
For a single certificate ($m=1$), feasibility is equivalent to
\begin{equation}
u_{\max}
\left\|G(x)^\top\nabla\varphi_1(x)\right\|_2
\ge
(\lambda_1-\beta_1)\varphi_1(x).
\label{eq:single_certificate_feasibility}
\end{equation}
\end{theorem}

\begin{proof}
The drift bound
$\mathcal K_f\varphi_i\le\lambda_i\varphi_i$
shows that
\[
a_i(x)u\le-d_i(x)
\]
is sufficient to enforce
\[
\mathcal K_f\varphi_i(x)
+a_i(x)u
\le
\beta_i\varphi_i(x).
\]
The stacked sufficient constraints are therefore
\eqref{eq:rate_shaping_feasibility}. For any fixed $u$, they hold
if and only if
\[
\max_i\{a_i(x)u+d_i(x)\}\le0,
\]
which yields \eqref{eq:minmax_feasibility} after minimizing over
$u\in\mathcal U$.

For the norm-bounded input set, if a unit vector $w(x)$ satisfies
\eqref{eq:common_descent_direction}, then
$u=u_{\max}w(x)$ satisfies
$B(x)u\preceq-d(x)$.

For $m=1$,
\[
\min_{\|u\|_2\le u_{\max}}a_1(x)u
=
-u_{\max}\|a_1(x)\|_2.
\]
Hence feasibility is equivalent to
$u_{\max}\|a_1(x)\|_2\ge d_1(x)$.
Using
\[
\|a_1(x)\|_2
=
\|G(x)^\top\nabla\varphi_1(x)\|_2
\]
gives \eqref{eq:single_certificate_feasibility}.\qed
\end{proof}

Theorem~\ref{thm:vector_qp_feasibility_geometric} shows that rate shaping
depends not only on the open-loop rates $\lambda_i$, but also on
the alignment of the certificate gradients with the control
distribution. In particular, if
$G(x)^\top\nabla\varphi_i(x)=0$, the input cannot instantaneously
modify the $i$th certificate. For multiple certificates, feasibility
requires sufficient control authority to provide a common descent
direction satisfying all certificate constraints. For stabilization,
the desired rates are additionally chosen such that $\beta_i<0$.

% Theorem~\ref{thm:vector_qp_feasibility_geometric} shows that feasibility is not
% determined only by the open-loop rates $\lambda_i$, but also by the alignment
% between certificate gradients and the control distribution. If
% $G(x)^\top\nabla\varphi_i(x)=0$, then the input cannot instantaneously affect
% the $i$th certificate. For underactuated systems, where $\mathrm{Range}(G(x))$
% is low-dimensional, the vector certificates must admit a common descent
% direction in input space. This observation motivates constructing
% supereigenfunctions using control-aware rollout dynamics so that their gradients are better aligned with the available control directions. This will be a topic of our future investigation and research. 

This formulation yields a convex quadratic program at each state. 
The certificates $\{\varphi_i\}$ act as a vector Lyapunov function, 
while the control input shapes their associated growth rates from 
open-loop values $\lambda_i$ to desired negative values $\beta_i$. 
Unlike classical CLF constructions, these certificates arise from 
Koopman-based operator representations and can be computed directly 
from trajectory data, providing a data-driven pathway for stabilizing 
control synthesis.

\subsection{Safety Characterization and Safe Control via Supereigenfunctions}
\label{subsec:safety}
%=============================================================================

We now use resolvent supereigenfunctions to construct dynamics-aware safety certificates. The key idea is to convert state constraints into a nonnegative \emph{risk probe} and to aggregate this risk along trajectories. The resulting certificate measures future risk-to-go rather than instantaneous constraint violation, and can be used both for safety characterization and convex safe-control synthesis.

%-----------------------------------------------------------------------------
\subsubsection{Risk probes and resolvent safety certificates}
%-----------------------------------------------------------------------------

Let the safe set, $\cal S$, and the unsafe set $\cal U$ be defined as 
\begin{equation}
\mathcal S:=\{x\in\Omega:\ g_i(x)\le 0,\ i=1,\dots,q\},
\qquad
\mathcal U:=\Omega\setminus\mathcal S,
\label{eq:safe_unsafe_sets}
\end{equation}
where \(g_i:\Omega\to\mathbb R\) are continuous constraint functions.  
Define a nonnegative risk map
\begin{equation}
r(x):=\sum_{i=1}^q \rho\!\big(g_i(x)\big),
\label{eq:risk_map_general}
\end{equation}
where \(\rho:\mathbb R\to\mathbb R_{\ge 0}\) is nonnegative and typically monotone increasing. Typical choices include
\begin{align}
\rho(s)&=\mathbf 1_{\{s>0\}} ,\label{eq:rho_indicator}\\
\rho(s)&=\max\{0,s\}, \label{eq:rho_hinge}\\
\rho(s)&=\log(1+e^{\kappa s}),\quad \kappa>0, \label{eq:rho_softplus}\\
\rho(s)&=e^{\kappa s},\quad \kappa>0. \label{eq:rho_exp}
\end{align}
The choice of $\rho$ determines the type of safety information
encoded by the certificate. The indicator probe is zero on the safe
set and uniformly positive on the unsafe set, but is discontinuous.
The hinge probe is continuous and zero on the safe set, but its value
approaches zero near the safety boundary. Softplus and exponential
probes provide smooth graded measures of risk but remain strictly
positive on both sides of the boundary. These distinctions determine
whether the resulting certificate supports occupation-time, hard
safety, or graded-risk interpretations.
% The first two give hard separation of safe and unsafe regions, while the latter two encode graded risk. The functions \eqref{eq:rho_indicator}--\eqref{eq:rho_exp} vanish (or remain small) on the safe set and become positive in the unsafe set, with the transition across the boundary ranging from discontinuous (indicator) to smooth (softplus and exponential).
For a given discount \(\lambda\ge 0\), define the infinite-horizon resolvent certificate
\begin{equation}
\varphi_\lambda(x):=
\int_0^\infty e^{-\lambda t}\,r(s_t(x))\,dt,
\label{eq:resolvent_risk_infinite}
\end{equation}
whenever the integral converges. For a finite horizon \(T>0\), 
\begin{equation}
\varphi_{\lambda,T}(x):=
\int_0^T e^{-\lambda t}\,r(s_t(x))\,dt.
\label{eq:resolvent_risk_finite}
\end{equation}

% \begin{proposition}[Resolvent safety inequality]
% \label{prop:resolvent_safety_ineq}
% Assume \(\varphi_\lambda\in C^1(\Omega)\) and \eqref{eq:resolvent_risk_infinite} is finite. Then
% \begin{equation}
% (\lambda I-\mathcal K_f)\varphi_\lambda=r,
% \qquad\text{hence}\qquad
% \mathcal K_f\varphi_\lambda=\lambda\varphi_\lambda-r\le \lambda\varphi_\lambda.
% \label{eq:resolvent_safety_generator}
% \end{equation}
% \end{proposition}

Thus \(\varphi_\lambda\) is a supereigenfunction of the drift dynamics and can be interpreted as a discounted risk-to-go. The parameter \(\lambda\) plays three roles: it ensures convergence, controls the time horizon of the certificate, and interpolates between long-horizon exposure (\(\lambda\) small) and near-term risk (\(\lambda\) large).

%-----------------------------------------------------------------------------
\subsubsection{Safety interpretation}
%-----------------------------------------------------------------------------

For any \(\lambda\ge 0\) and \(T\in(0,\infty]\), the sublevel set
\begin{equation}
\mathcal S_{\lambda,T}(c):=\{x\in\Omega:\ \varphi_{\lambda,T}(x)\le c\}
\label{eq:soft_safe_set}
\end{equation}
defines a dynamics-aware safe set: the future accumulated risk along the open-loop trajectory is at most \(c\). This yields a natural notion of \emph{soft safety}. To convert this certificate into quantitative violation bounds, we make following assumption.

\begin{assumption}[Unsafe-set lower bound]
\label{ass:unsafe_lower_bound}
There exists \(\underline r>0\) such that \(r(x)\ge \underline r\) for all \(x\in\mathcal U\).
\end{assumption}
\begin{theorem}[Discounted unsafe-occupation bound]
\label{thm:unsafe_occupation}
Under Assumption~\ref{ass:unsafe_lower_bound}, for any \(x\in \mathcal S_{\lambda,T}(c)\),
the discounted time spent in the unsafe set \(\mathcal U\) over the horizon \(T\) satisfies
\begin{align}
\int_0^T e^{-\lambda t}\,
\mathbf 1_{\mathcal U}(s_t(x))\,dt
\le
\frac{c}{\underline r}.\label{bound}
\end{align}
\end{theorem}
% \begin{theorem}[Discounted unsafe-occupation bound]
% \label{thm:unsafe_occupation}
% Under Assumption~\ref{ass:unsafe_lower_bound}, for any \(T\in(0,\infty]\) and \(\lambda\ge 0\),
% \begin{equation}
% \int_0^T e^{-\lambda t}\,\mathbf 1_{\mathcal U}(s_t(x))\,dt
% \le
% \frac{1}{\underline r}\,\varphi_{\lambda,T}(x),
% \qquad \forall x\in\Omega.
% \label{eq:unsafe_occ_bound}
% \end{equation}
% In particular, if \(x\in\mathcal S_{\lambda,T}(c)\), then the discounted time spent in \(\mathcal U\) up to time \(T\) is at most \(c/\underline r\).
% \end{theorem}
\begin{proof}
By Assumption~\ref{ass:unsafe_lower_bound}, $
r(x)\ge \underline r\,\mathbf 1_{\mathcal U}(x).
$
Therefore,
\[
\varphi_{\lambda,T}(x)
=
\int_0^T e^{-\lambda t}r(s_t(x))\,dt
\ge
\underline r
\int_0^T e^{-\lambda t}
\mathbf 1_{\mathcal U}(s_t(x))\,dt.
\]
If \(x\in\mathcal S_{\lambda,T}(c)\), then $
\varphi_{\lambda,T}(x)\le c,
$
which gives (\ref{bound}).\qed
\end{proof}
% \begin{proof}
% Since \(r\ge \underline r\,\mathbf 1_{\mathcal U}\), multiply by \(e^{-\lambda t}\) and integrate over \([0,T]\).
% \end{proof}

For finite-horizon hard safety, assume additionally:

\begin{assumption}[Zero-safe positive-unsafe probe]
\label{ass:zero_safe_positive_unsafe}
The probe \(r:\Omega\to\mathbb R_{\ge 0}\) is continuous and satisfies
\[
r(x)=0 \ \text{for all } x\in\mathcal S,
\qquad
r(x)>0 \ \text{for all } x\in\mathcal U.
\]
\end{assumption}
\begin{theorem}[Finite-horizon hard safety]
\label{thm:finite_hard_safety}
Under Assumption~\ref{ass:zero_safe_positive_unsafe}, for any finite \(T>0\) and any \(\lambda\ge 0\),
\[
\varphi_{\lambda,T}(x)=0
\quad\Longleftrightarrow\quad
s_t(x)\in\mathcal S,\ \forall t\in[0,T].
\]
\end{theorem}
\begin{proof}
If \(s_t(x)\in\mathcal S\) for all \(t\in[0,T]\), then
\(r(s_t(x))=0\) for all \(t\in[0,T]\). Hence
\[
\varphi_{\lambda,T}(x)
=
\int_0^T e^{-\lambda t} r(s_t(x))\,dt
=0.
\]

Conversely, suppose \(\varphi_{\lambda,T}(x)=0\). Since
\(e^{-\lambda t}r(s_t(x))\ge 0\), it follows that
\[
r(s_t(x))=0
\quad \text{for a.e. } t\in[0,T].
\]
Because \(r\) and \(s_t(x)\) are continuous, the map
\(t\mapsto r(s_t(x))\) is continuous. Therefore, if it were positive at any
time \(t_0\in[0,T]\), it would be positive on a neighborhood of \(t_0\),
contradicting the fact that it is zero almost everywhere. Hence
\[
r(s_t(x))=0
\quad \forall t\in[0,T].
\]
By the assumption that \(r(x)=0\) on \(\mathcal S\) and \(r(x)>0\) on
\(\mathcal U\), this implies
$
s_t(x)\in\mathcal S,
\qquad \forall t\in[0,T].
$\qed
\end{proof}

% \begin{assumption}[Hard-separation probe]
% \label{ass:hard_separation}
% The probe satisfies \(r(x)=0\) for all \(x\in\mathcal S\) and \(r(x)\ge \underline r>0\) for all \(x\in\mathcal U\).
% \end{assumption}

% \begin{theorem}[Finite-horizon hard safety]
% \label{thm:finite_hard_safety}
% Under Assumption~\ref{ass:hard_separation}, for any finite \(T>0\) and any \(\lambda\ge 0\),
% \begin{equation}
% \varphi_{\lambda,T}(x)=0
% \quad\Longleftrightarrow\quad
% s_t(x)\in\mathcal S,\ \forall t\in[0,T].
% \label{eq:hard_finite_equiv}
% \end{equation}
% \end{theorem}

% \begin{proof}
% If \(s_t(x)\in\mathcal S\) for all \(t\in[0,T]\), then \(r(s_t(x))\equiv 0\), so \(\varphi_{\lambda,T}(x)=0\). Conversely, if \(\varphi_{\lambda,T}(x)=0\), then \(r(s_t(x))=0\) almost everywhere, hence by continuity \(s_t(x)\in\mathcal S\) for all \(t\in[0,T]\).
% \end{proof}

Thus, depending on the probe \(r\), the certificate \(\varphi_{\lambda,T}\) supports either soft safety, occupation-time bounds, or finite-horizon hard safety.

%-----------------------------------------------------------------------------
\subsubsection{Safe control synthesis}
%-----------------------------------------------------------------------------

We now turn these drift-based safety certificates into convex feedback constraints. Consider the control-affine system
\begin{equation}
\dot x=f(x)+G(x)u,
\qquad x\in\Omega\subset\mathbb R^n,\ \ u\in\mathbb R^p.
\label{eq:safety_control_system}
\end{equation}

To enforce safety, it is not necessary to render all sublevel sets of
\(\varphi_\lambda\) invariant. Instead,
we select a safety budget \(c>0\) and enforce invariance only of the chosen
risk sublevel set
\[
\mathcal S_\lambda(c)
:=
\{x\in\Omega:\ \varphi_\lambda(x)\le c\}.
\]
Equivalently, define a function
$
h_\lambda(x):=c-\varphi_\lambda(x),
$
so that
\[
\mathcal S_\lambda(c)=\{x:\ h_\lambda(x)\ge0\}.
\]
For a certificate \(\varphi_\lambda\) constructed from the resolvent formula
\eqref{eq:resolvent_risk_infinite}, 
Theorem~\ref{thm:d2c_resolvent}
gives
\begin{equation}
({\mathcal K}_f\varphi_\lambda)(x)
=
\lambda\varphi_\lambda(x)-r(x).
\label{eq:safety_resolvent_id}
\end{equation}
Hence, under the controlled dynamics
\eqref{eq:safety_control_system},
\begin{equation}
\dot\varphi_\lambda(x)
=
\lambda\varphi_\lambda(x)-r(x)
+
({\mathcal K}_G\varphi_\lambda)(x)\,u.
\label{eq:safety_phi_dot_control}
\end{equation}

A sufficient condition for forward invariance of
\(\mathcal S_\lambda(c)\) is
\begin{equation}
\dot h_\lambda(x)\ge -\alpha h_\lambda(x),
\qquad \alpha>0,
\label{eq:cbf_condition_phi}
\end{equation}
or, equivalently,
\begin{equation}
\dot\varphi_\lambda(x)
\le
\alpha\big(c-\varphi_\lambda(x)\big).
\label{eq:soft_barrier_condition}
\end{equation}
Substituting \eqref{eq:safety_phi_dot_control} into
\eqref{eq:soft_barrier_condition} yields the affine constraint
\begin{equation}
({\mathcal K}_G\varphi_\lambda)(x)\,u
\le
\alpha\big(c-\varphi_\lambda(x)\big)
-\lambda\varphi_\lambda(x)
+r(x).
\label{eq:soft_barrier_affine}
\end{equation}

\begin{remark}[Robust safety with approximate D2C certificates]
For a finite-rollout or data-driven certificate satisfying
\[
\mathcal K_f\widehat\phi_\lambda
\le
\lambda\widehat\phi_\lambda-r+\varepsilon,
\]
forward invariance of the selected sublevel set
$\{\widehat\phi_\lambda\le c\}$ is ensured by the tightened constraint
\[
(\mathcal K_G\widehat\phi_\lambda)u
\le
\alpha(c-\widehat\phi_\lambda)
-\lambda\widehat\phi_\lambda+r-\varepsilon.
\]
For the D2C construction of Theorem~\ref{thm:d2c_resolvent},
$\varepsilon$ is determined explicitly by the finite-rollout and
certificate approximation errors.
\end{remark}
This formulation applies to both hard and graded-risk probe functions.
In particular, if $r=r_{\mathrm{hard}}$ satisfies
\begin{equation}
r_{\mathrm{hard}}(x)=0 \ \text{on}\ \mathcal S,
\qquad
r_{\mathrm{hard}}(x)>0 \ \text{on}\ \mathcal U,
\label{eq:hard_probe_sep}
\end{equation}
then the resulting certificate admits the hard-safety interpretation
of Theorem~\ref{thm:finite_hard_safety}, whereas smooth positive probes provide
graded measures of future risk. Under feedback, the affine constraint
above renders the selected risk sublevel set
$\{\widehat\varphi_{\lambda}\le c\}$ forward invariant; this implies
hard safety with respect to $\mathcal S$ whenever the selected sublevel
set is contained in $\mathcal S$.

For one or more certificates $\widehat\varphi^i_{\lambda_i}$ satisfying
\[
\mathcal K_f\widehat\varphi^i_{\lambda_i}
\le
\lambda_i\widehat\varphi^i_{\lambda_i}-r_i+\varepsilon_i,
\]
the corresponding tightened affine constraints can be stacked and
enforced through the convex QP
\begin{equation}
\begin{aligned}
u^\star(x)\in\arg\min_{u\in\mathbb R^p}\quad &
\frac12\,(u-u_{\mathrm{nom}}(x))^\top
R\,(u-u_{\mathrm{nom}}(x)) \\
\text{s.t.}\quad &
A_{\mathrm{safe}}(x)u
\preceq b_{\mathrm{safe}}(x),\\
&u\in\mathcal U,
\end{aligned}
\label{eq:safety_qp_two_phi}
\end{equation}
where $R\succ0$, $u_{\mathrm{nom}}$ is a nominal controller, and the
$i$th row of the constraint is
\begin{equation}
(\mathcal K_G\widehat\varphi^i_{\lambda_i})(x)u
\le
\alpha_i\!\left(c_i-\widehat\varphi^i_{\lambda_i}(x)\right)
-\lambda_i\widehat\varphi^i_{\lambda_i}(x)
+r_i(x)-\varepsilon_i.
\label{eq:safe_constraint_approx}
\end{equation}
For exact resolvent certificates, $\varepsilon_i=0$; for finite-rollout
or data-driven certificates, $\varepsilon_i$ accounts for the
corresponding D2C residual.
% \begin{theorem}[Forward invariance of certificate sublevel sets] \label{thm:safety_invariance_qp} Assume the QP \eqref{eq:safety_qp_two_phi} is feasible for all \(x\) in the operating region and \(u^\star(x)\) is any measurable selection. Then every sublevel set enforced through a constraint of the form \eqref{eq:soft_barrier_affine} is forward invariant under the closed-loop dynamics. In particular, intersections of such sublevel sets are also forward invariant. \end{theorem} \begin{proof} Each enforced inequality implies the corresponding barrier differential inequality \eqref{eq:soft_barrier_condition}. Standard comparison arguments yield forward invariance of the associated sublevel sets, and intersections preserve invariance. \end{proof}

\begin{theorem}[Forward invariance of selected risk sublevel sets]
\label{thm:safety_invariance_qp}
Let \(\varphi^i_{\lambda_i}:\Omega\to\mathbb R_{\ge 0}\), 
\(i=1,\ldots,N_\phi\), be continuously differentiable safety certificates
constructed from the resolvent formula, with associated risk probes \(r_i\),
discounts \(\lambda_i\), and safety budgets \(c_i>0\). Define
\[
\mathcal S_i(c_i)
:=
\{x\in\Omega:\ \varphi^i_{\lambda_i}(x)\le c_i\}.
\]
Assume that the QP \eqref{eq:safety_qp_two_phi} is feasible for all \(x\)
in the operating region and that \(u^\star(x)\) is a measurable selection
of its solution. Suppose that, for each enforced certificate, the selected
control satisfies for $\alpha_i>0$
\begin{equation}
({\mathcal K}_G\varphi^i_{\lambda_i})(x)\,u^\star(x)
\le
\alpha_i\big(c_i-\varphi^i_{\lambda_i}(x)\big)
-\lambda_i\varphi^i_{\lambda_i}(x)
+r_i(x).
\label{eq:safety_affine_condition_i}
\end{equation}
Then each selected risk sublevel set \(\mathcal S_i(c_i)\) is forward
invariant under the closed-loop dynamics. Consequently, any finite
intersection
$
\mathcal S
=
\bigcap_{i=1}^{N_\phi}\mathcal S_i(c_i)
$
of enforced risk sublevel sets is forward invariant.
\end{theorem}

\begin{proof}
Fix \(i\in\{1,\ldots,N_\phi\}\) and define $
h_i(x):=c_i-\varphi^i_{\lambda_i}(x).
$
By the resolvent identity,
\[
({\mathcal K}_f\varphi^i_{\lambda_i})(x)
=
\lambda_i\varphi^i_{\lambda_i}(x)-r_i(x).
\]
Along the closed-loop dynamics,
\[
\dot\varphi^i_{\lambda_i}(x)
=
\lambda_i\varphi^i_{\lambda_i}(x)-r_i(x)
+
({\mathcal K}_G\varphi^i_{\lambda_i})(x)u^\star(x).
\]
Using \eqref{eq:safety_affine_condition_i}, we obtain
$
\dot\varphi^i_{\lambda_i}(x)
\le
\alpha_i\big(c_i-\varphi^i_{\lambda_i}(x)\big).
$
Since \(h_i=c_i-\varphi^i_{\lambda_i}\), this is equivalent to
$
\dot h_i(x)\ge -\alpha_i h_i(x).
$
By the comparison lemma,
\[
h_i(x(t))\ge e^{-\alpha_i t}h_i(x(0)).
\]
Thus, if \(x(0)\in\mathcal S_i(c_i)\), then \(h_i(x(0))\ge0\), and hence
\(h_i(x(t))\ge0\) for all \(t\ge0\). Therefore
$
\varphi^i_{\lambda_i}(x(t))\le c_i,
$
so \(\mathcal S_i(c_i)\) is forward invariant.
The same argument holds for every certificate under the same
closed-loop control, any finite intersection of such invariant sets is also
forward invariant.\qed
\end{proof}
This construction is closely related in form to CBF-QP methods, but the certificate \(\varphi_\lambda\) is constructed from the drift dynamics and therefore encodes accumulated future risk rather than an instantaneous geometric barrier.

\subsection{Contraction and Incremental Stability via Directional Supereigenfunctions}
\label{subsec:contraction}
%=============================================================================

We now interpret the directional supereigenfunctions constructed in the previous section through the lens of contraction and incremental stability. While state-space supereigenfunctions yield Lyapunov-type certificates for convergence to equilibria, tangent-space supereigenfunctions naturally characterize \emph{convergence between trajectories}.

Consider two trajectories $x(t)=s_t(x_0)$ and $\tilde{x}(t)=s_t(\tilde{x}_0)$, and define the differential displacement $v(t)=\tilde{x}(t)-x(t)$ governed locally by the variational dynamics
\begin{equation}
\dot v = Df(x)v.
\end{equation}
A directional supereigenfunction $\varphi(x,v)$ provides an upper bound on the growth of such displacements. In particular, for a quadratic observable of the form $
\varphi_M(x,v)=v^\top M(x)v,
$, where the matrix $M(x)$ can be obtained using the MET or resolvent-based scheme as given in Eq. (\ref{eq:met_phi_iv}) and (\ref{eq:gramian_metric_family}) respectively. The supereigenfunction condition
\begin{equation}
\dot \varphi_M \le \lambda \varphi_M
\label{eq:contraction_supereig_short}
\end{equation}
implies
\[
\varphi_M(x(t),v(t)) \le e^{\lambda t}\varphi_M(x_0,v_0).
\]
When $\lambda<0$, this yields exponential convergence of nearby trajectories, recovering the classical notion of contraction. Thus, contraction metrics arise as a special case of directional supereigenfunctions. The constructions process developed in Section \ref{subsec:met_directional}  provide two complementary ways of obtaining such certificates. The MET-based construction yields directional observables
\[
\varphi_i(x,v)=|P_i(x)v|^2,
\]
which characterize anisotropic growth along Oseledets subspaces, with rates given by twice the Lyapunov exponents. These provide intrinsic, data-driven estimates of expansion and contraction directions. In contrast, the Gramian-based construction produces a state-dependent metric $M(x)$ through trajectory integrals of the variational dynamics. The associated observable $\varphi_M(x,v)$ satisfies a supereigenfunction inequality of the form \eqref{eq:contraction_supereig_short}, yielding a finite-time or discounted certificate of incremental stability. Importantly, this construction is computable directly from data and remains well-defined even in regimes where asymptotic quantities such as Lyapunov exponents are difficult to estimate.

% \medskip

% From this perspective, directional supereigenfunctions provide a unified framework for incremental analysis. The MET-based constructions capture asymptotic growth rates, while Gramian-based constructions yield computable finite-time contraction metrics. Both lead naturally to inequality constraints that can be enforced for control design, enabling convex synthesis of stabilizing,  and contracting controllers.

%-----------------------------------------------------------------------------
\subsubsection{Contraction-based control design}
%-----------------------------------------------------------------------------
%=============================================================================

We now extend the certificate-based control design of the previous section to
the incremental (contraction) setting. Instead of state-space certificates, we
consider directional supereigenfunctions defined on the tangent bundle. Consider the control-affine system $\dot x=f(x)+G(x)u$ and the associated
variational dynamics
\[
\dot v = \left(Df(x) + D G(x)u\right)v.
\]
Let $\varphi(x,v)=v^\top M(x)v$ be a  supereigenfunction constructed
via the Gramian or MET-based approaches described earlier. Let $\Psi(x,v)$ be the vector of supereigenfunctions on the tangent bundle. A contraction condition is obtained by enforcing
\begin{equation}
\dot \Psi(x,v)
\preceq B \Psi(x,v),
\;\;\; B={\rm diag}(\beta_i),\;\;\beta_i<0
\label{eq:contraction_control_condition}
\end{equation}
for $i=1,\dots,m$ which ensures exponential decay of differential displacements and hence
incremental stability.
Proceeding as in the stabilization case, the derivative of $\varphi$ along 
the tangent dynamics can be written as
\[
\dot \Psi(x,v)
=
({\mathcal K}_f^{\tan} \Psi)(x,v) + ({\mathcal K}^{\tan}_{G} \Psi)(x,v)u,
\]
% \[
% ({\mathcal K}^{\tan}_{Gu}\Psi)(x,v)
% :=
% ({\mathcal K}^{\tan}_{G}\Psi)(x,v)\,u.
% \]
where,
\[
({\mathcal K}^{\tan}_{G}\Psi)(x,v)
=
\begin{bmatrix}
({\mathcal K}^{\tan}_{G}\psi_1)(x,v)\\
\vdots\\
({\mathcal K}^{\tan}_{G}\psi_m)(x,v)
\end{bmatrix},
\]
which is affine in $u$. Assume the open-loop growth satisfies $
({\mathcal K}_f^{\tan} \Psi)(x,v)\preceq \Lambda\,\Psi(x,v)$ with $\Lambda={\rm diag}(\lambda_i)$.
% for some constant $\lambda\in\mathbb R$. The Lie derivative, ${\mathcal L}_f^\Delta$ of $\varphi$ along these dynamics is defined as
% \begin{equation}
% \mathcal L_f^\Delta \varphi(x,v)
% :=
% \nabla_x \varphi(x,v)^\top f(x)
% +
% \nabla_v \varphi(x,v)^\top Df(x)v.
% \label{eq}
% \end{equation}
A sufficient condition for \eqref{eq:contraction_control_condition} is
\begin{equation}
({\mathcal K}_{G}^{\tan} \Psi )(x,v) u
\preceq
-(\Lambda-B)\,\Psi(x,v),
\label{eq:contraction_affine}
\end{equation}
with $\beta_i<0$. This yields the pointwise convex quadratic program
\begin{equation}
\begin{aligned}
u^\star(x)\in&\arg\min_u\quad  \tfrac12 u^\top R u \\
\text{s.t.}\quad &
({\mathcal K}_{G}^{\tan} \Psi)(x,v)u \preceq-(\Lambda-B)\,\Psi(x,v), 
\; u\in\mathcal U,
\end{aligned}
\label{eq:qp_contraction}
\end{equation}
which enforces contraction of the directional certificate.
\medskip

This formulation can be interpreted as shaping the differential growth rate of
the system from its open-loop value $\Lambda$ to a desired negative rate $B$,
thereby enforcing synchronization and tracking between trajectories.
Compared to classical contraction metrics, the proposed framework provides a
data-driven pathway for constructing and enforcing contraction certificates
via Koopman supereigenfunctions.

%=============================================================================

\section{Computational Construction of Directional Supereigenfunctions}
\label{sec:computation}
%%%%%%%%%%%%%%%%%%%%%%%%%%%%%%%%%%%%%%%%%%%%%%%%%%%%%%%%%%%%%%%%%%%%%%%%%%%%%%

We summarize data-driven procedures for constructing directional 
supereigenfunctions from trajectory samples, avoiding high-dimensional PDEs. 
We present three complementary approaches: 1) MET/QR: intrinsic directional structure from tangent data; 2) Gramian: tangent-space supereigenfunctions from trajectory sensitivities; 3) Resolvent: state-space certificates from observable rollouts.

%=============================================================================
\subsection{MET/QR: intrinsic directional certificates}
\label{subsec:met_qr}
%=============================================================================

The MET implies that along typical trajectories there exist Lyapunov exponents 
$\chi_i$ and associated invariant directions. These can be estimated via the 
incremental QR (Benettin) iteration applied to the tangent cocycle. Given a trajectory $\{x_k\}_{k=0}^{N}$ with step $\Delta t$, and tangent maps
\[
A_k := Ds_{\Delta t}(x_k),
\]
perform
\begin{equation}
A_k Q_k = Q_{k+1}R_k,
\label{eq:qr_incremental_comp}
\end{equation}
with $Q_k$ orthonormal and $R_k$ upper triangular. Then
\begin{equation}
\widehat{\chi}_i
=
\frac{1}{N\Delta t}\sum_{k=0}^{N-1}\log (R_k)_{ii},
\qquad
\widehat{\lambda}_i = 2\widehat{\chi}_i.
\label{eq:chi_hat_comp}
\end{equation}

The columns $q_i(x_k)$ approximate Oseledets directions, yielding projectors $
P_i(x_k)=q_i(x_k)q_i(x_k)^\top$.
Define the directional observable
\[
\varphi_i(x_k,v_k) = (q_i(x_k)^\top v_k)^2,
\]
where $v_k$ evolves according to the tangent dynamics
$
v_{k+1} = A_k v_k
$.
Then $\varphi_i$ grows asymptotically at rate $\widehat{\lambda}_i$.
% Define the directional observable
% \[
% \varphi_i(x_k)=(q_i(x_k)^\top x_k)^2,
% \]
% which grows asymptotically at rate $\widehat{\lambda}_i$.

% \noindent {\it Algorithm}:
% \begin{enumerate}
% \item Collect trajectory $\{x_k\}$ and tangent maps $A_k$.
% \item Run QR iteration \eqref{eq:qr_incremental_comp}.
% \item Estimate $\widehat{\lambda}_i=2\widehat{\chi}_i$.
% \item Form $\varphi_i(x_k,v_k)=(q_i^\top v_k)^2$ and regress/interpolate.
% \end{enumerate}

%=============================================================================
\subsection{Gramian: tangent-space supereigenfunctions}
\label{subsec:gramian_comp}
%=============================================================================

Directional supereigenfunctions on the tangent bundle can be constructed 
directly from trajectory sensitivities. Given $Q_i:\Omega\to\mathbb S_+^n$ and $\lambda_i>0$, define
\begin{equation}
\varphi_i(x_k,v)
\approx
\Delta t\sum_{j=0}^{H}
e^{-2\lambda_i j\Delta t}
\|Q_i^{1/2}(x_{k+j})\,Y_{k,j}v\|^2,
\label{eq:gramian_sum_comp}
\end{equation}
where $Y_{k,j}\approx Ds_{j\Delta t}(x_k)$ is obtained from variational simulation. This approximates the infinite-horizon Gramian (\ref{eq:gramian_phi_family}),
% \[
% \varphi_i(x,v)
% =
% \int_0^\infty e^{-2\lambda_i t}
% \|Q_i^{1/2}(s_t(x))Y(t,x)v\|^2 dt,
% \]
which satisfies
\[
\mathcal K_{\mathrm{tan}}\varphi_i \le 2\lambda_i \varphi_i.
\]

% \noindent {\it Algorithm}:\begin{enumerate}
% \item Simulate trajectories and tangent dynamics to obtain $Y_{k,j}$.
% \item Choose $Q_i$ to emphasize desired directions.
% \item Compute $\varphi_i(x_k,v)$ via \eqref{eq:gramian_sum_comp}.
% \item Optionally regress $M_i(x)$ such that $\varphi_i(x,v)=v^\top M_i(x)v$.
% \end{enumerate}

\begin{remark}[Comparison of approaches]
MET/QR yields intrinsic directions and rates; the resolvent method constructs 
state-space certificates directly from observable data; and the Gramian method 
provides constructive tangent-space supereigenfunctions. Together, these form a 
unified data-driven framework for directional analysis and control synthesis.
\end{remark}

%=============================================================================
\subsection{Resolvent: state-space supereigenfunctions}
\label{subsec:resolvent_comp}
%=============================================================================

Given nonnegative probes $g_i:\Omega\to\mathbb R_{\ge 0}$ and $\lambda_i>0$, define
\begin{equation}
\varphi_i(x)
=
\int_0^\infty e^{-\lambda_i t}g_i(s_t(x))dt,
\label{eq:resolvent_def_comp}
\end{equation}
which satisfies $
\mathcal K_f\varphi_i=\lambda_i\varphi_i-g_i \le \lambda_i\varphi_i.
$
With sampled rollouts $x_{k+j}\approx s_{j\Delta t}(x_k)$,
\begin{equation}
\varphi_i(x_k)
\approx
\Delta t\sum_{j=0}^{H} e^{-\lambda_i j\Delta t}g_i(x_{k+j}).
\label{eq:resolvent_sum_comp}
\end{equation}

The rates are estimated as 
\[
\rho_i^\star
=
\max_k \frac{\varphi_i(x_{k+1})}{\varphi_i(x_k)+\eta},
\qquad
\lambda_i^\star = \frac{1}{\Delta t}\log \rho_i^\star.
\]

% \noindent {\it Algorithm}:\begin{enumerate}
% \item Choose probes $g_i$ aligned with the task.
% \item Compute $\varphi_i(x_k)$ via \eqref{eq:resolvent_sum_comp}.
% \item Estimate $\lambda_i$ if needed.
% \item Fit $\varphi_i(\cdot)$ over $\Omega$.
% \end{enumerate}

\section{Simulation Results}

We illustrate the proposed framework on four examples. The Duffing oscillator is used to demonstrate directional supereigenfunctions, MET-based certificates, and uncertainty propagation. The inverted pendulum on a cart illustrates rollout-based stabilizing control. A leader–follower Duffing system demonstrates contraction-based synchronization. A four-dimensional Duffing system with obstacles highlights the unified treatment of stabilization and safety.
\subsection{Duffing oscillator}
We first consider the damped Duffing oscillator in the double-well regime
\begin{equation}
\dot x = y,\qquad
\dot y = -\delta y - \alpha x - \beta x^3,
\label{eq:duffing_ct}
\end{equation}
with $(x,y)\in\mathbb R^2$. We use
$
\alpha=-1,\; \beta=1,\; \delta=0.2,
$
sampling time $\Delta t=0.01$, horizon $T=50$ and a $121\times 121$ grid over
% $[-2,2]\times[-2,2]$.

\subsubsection{Resolvent-based supereigenfunctions}

We choose the probe functions
\[
g_1(x,y)=x^2,\qquad g_2(x,y)=y^2,
\qquad \lambda_1=\lambda_2=0,
\]
so the resolvent reduces to a finite-horizon accumulation
\begin{equation}
\varphi_i(z_0)\approx \Delta t\sum_{j=0}^{H} g_i(z_j),\qquad i=1,2,
\label{eq:duffing_resolvent_sum_settings_short}
\end{equation}
where $z_{j+1}$ is obtained by RK4 integration. The resulting certificates capture distinct aspects of the dynamics. The function $\varphi_2$ is closely aligned with the mechanical energy
$
E=\tfrac12 y^2+\tfrac{\alpha}{2}x^2+\tfrac{\beta}{4}x^4,
$
consistent with the identity $\dot E=-y^2$. The function $\varphi_1$ highlights the weakly dissipative direction associated with the phase-space geometry.

\subsubsection{MET-based computation}

We also compute directional supereigenfunctions using the MET. The variational dynamics
\[
\dot Y = J(x(t))Y,\qquad Y(0)=I,
\]
are integrated with QR re-orthonormalization at every step using $qStep=1$ and $\Delta t_{\rm QR}=0.01$. After discarding an initial transient of length $10$, Lyapunov exponents are estimated as
\[
\widehat{\chi}_i =
\frac{1}{\mathrm{countR}\,\Delta t_{\mathrm{QR}}}
\sum_{k=1}^{\mathrm{countR}} \log |R_{ii}(k)|,
\qquad
\widehat{\lambda}_i=2\widehat{\chi}_i .
\]

Numerically, we obtain $\widehat{\lambda}_i\approx -0.3$, consistent with the local linearization about the attractors $(\pm1,0)$, for which the corresponding Lyapunov exponents are approximately $-0.15$. This agreement supports the MET-based directional construction.

\subsubsection{Uncertainty propagation}

To illustrate set-valued prediction, we consider the uncertain initial set
\[
\mathcal X_0=\{z:\ z\in z_c+\mathcal B\},\qquad z_c=(0,0),
\]
where $\mathcal B$ is a square of side length $0.2$. Monte--Carlo samples from $\mathcal X_0$ are propagated under \eqref{eq:duffing_ct}. Using the certificates \eqref{eq:duffing_resolvent_sum_settings_short}, we assign the envelope rates
\begin{equation}
{\cal K}_f \varphi_1 \le \lambda_1 \varphi_1,\qquad
{\cal K}_f \varphi_2 \le \lambda_2 \varphi_2,
\qquad
\lambda_1=0.1,\ \lambda_2=-0.1.
\label{eq:duffing_two_rates_short}
\end{equation}
This yields expansion in the $x$-direction and contraction in the $y$-direction. Define
\[
\bar\varphi_i(0)=\sup_{z\in\mathcal X_0}\varphi_i(z),
\qquad
\bar\varphi_i(t)=e^{\lambda_i t}\bar\varphi_i(0).
\]
Then,
$
\varphi_i(z(t))\le \bar\varphi_i(t), i=1,2,
$
so the reachable envelope is
\[
\mathcal R_{\varphi}(t;\mathcal X_0)
=
\{z:\ \varphi_1(z)\le \bar\varphi_1(t),\ \varphi_2(z)\le \bar\varphi_2(t)\}.
\]
The simulated trajectories remain inside these envelopes, illustrating certified uncertainty propagation using only two supereigenfunction coordinates. We select two supereigenfunctions with discount rates 
$\pm 0.1$, reflecting the local expansion and contraction behavior near the saddle equilibrium around which the initial uncertainty set is defined. More generally, the choice of the number of supereigenfunctions, their associated discount rates, and the probing functions depends critically on factors such as the state-space dimension and the relative rates of expansion and contraction of trajectories.

\begin{figure}
\centering
\includegraphics[width=1.0\linewidth]{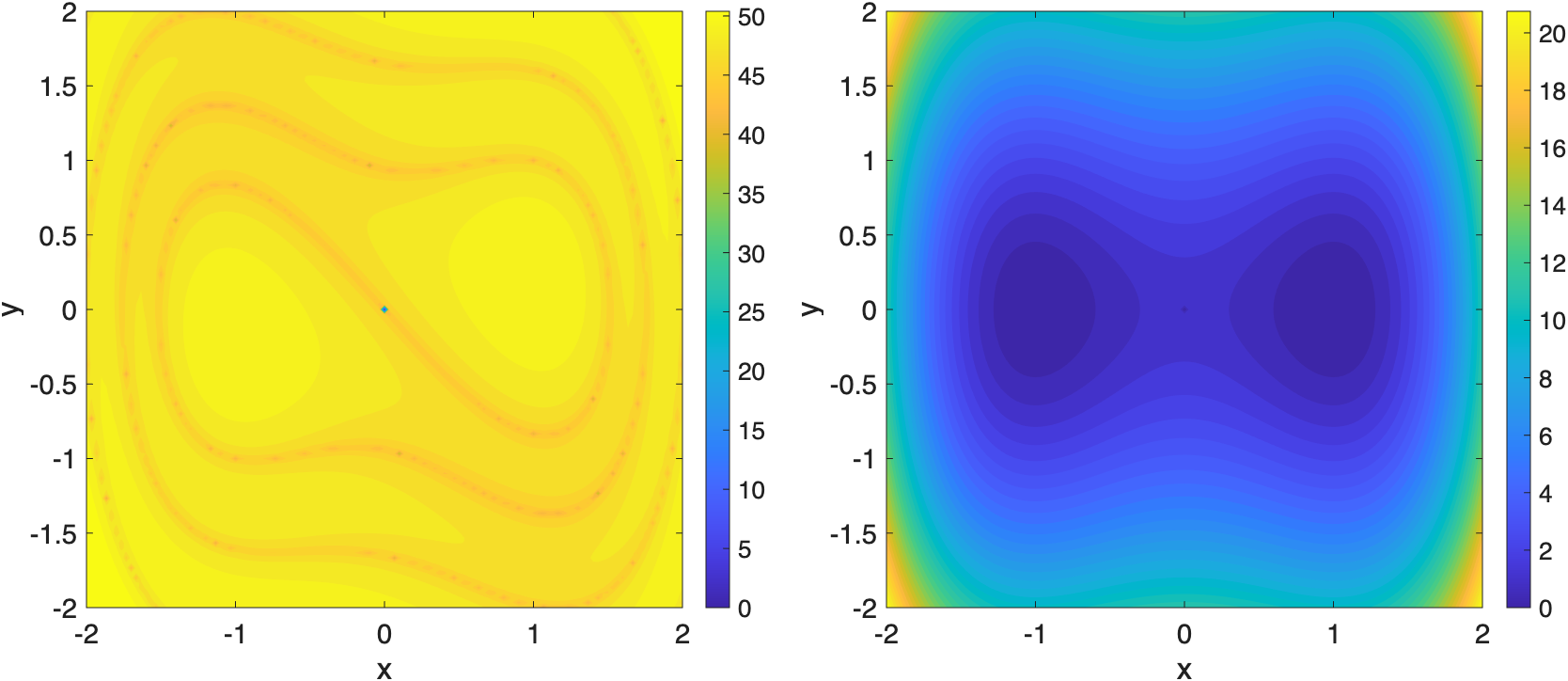}
\caption{Resolvent-based supereigenfunctions for the Duffing oscillator with probes $g_1=x^2$ (left) and $g_2=y^2$ (right). The two certificates capture distinct geometric features of the flow.}
\label{fig:duffing_resolvent_zero}
\end{figure}

\begin{figure}
\centering
\includegraphics[width=1.0\linewidth]{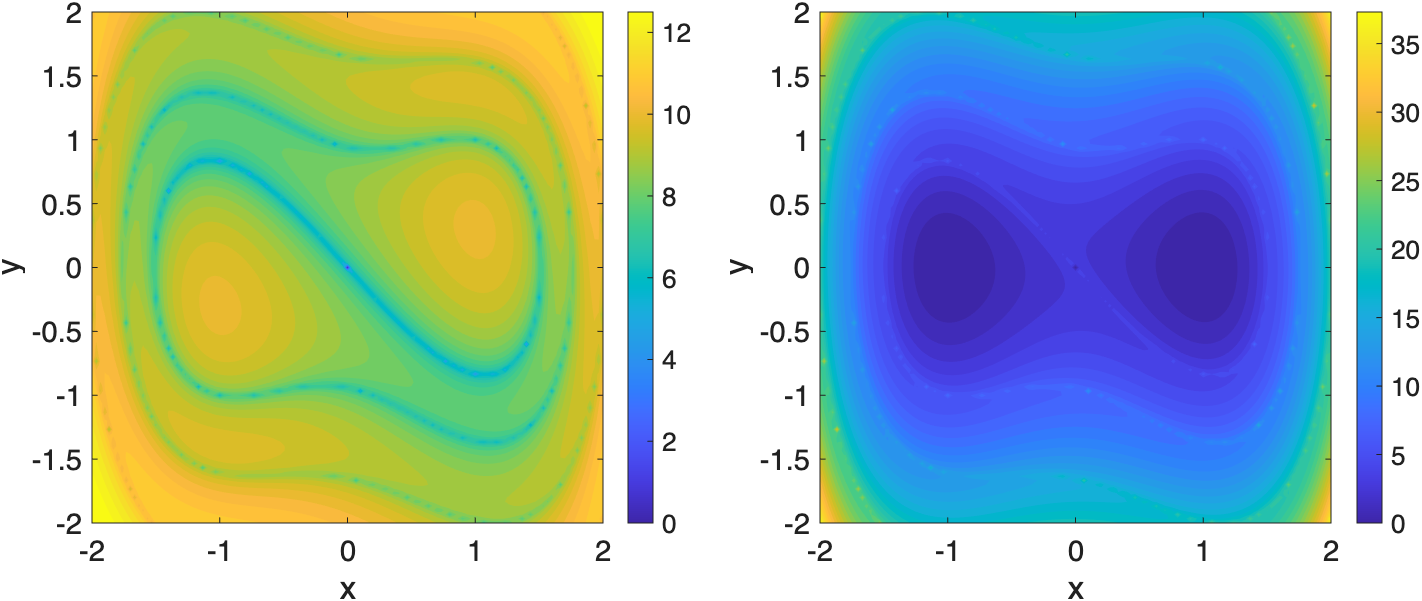}
\caption{Supereigenfunctions for the Duffing oscillator with nonzero envelope rates $\lambda_1=0.1$ and $\lambda_2=-0.1$, illustrating expanding and contracting directions.}
\label{fig:duffing_resolvent_shifted}
\end{figure}

\begin{figure}
\centering
\includegraphics[width=1.0\linewidth]{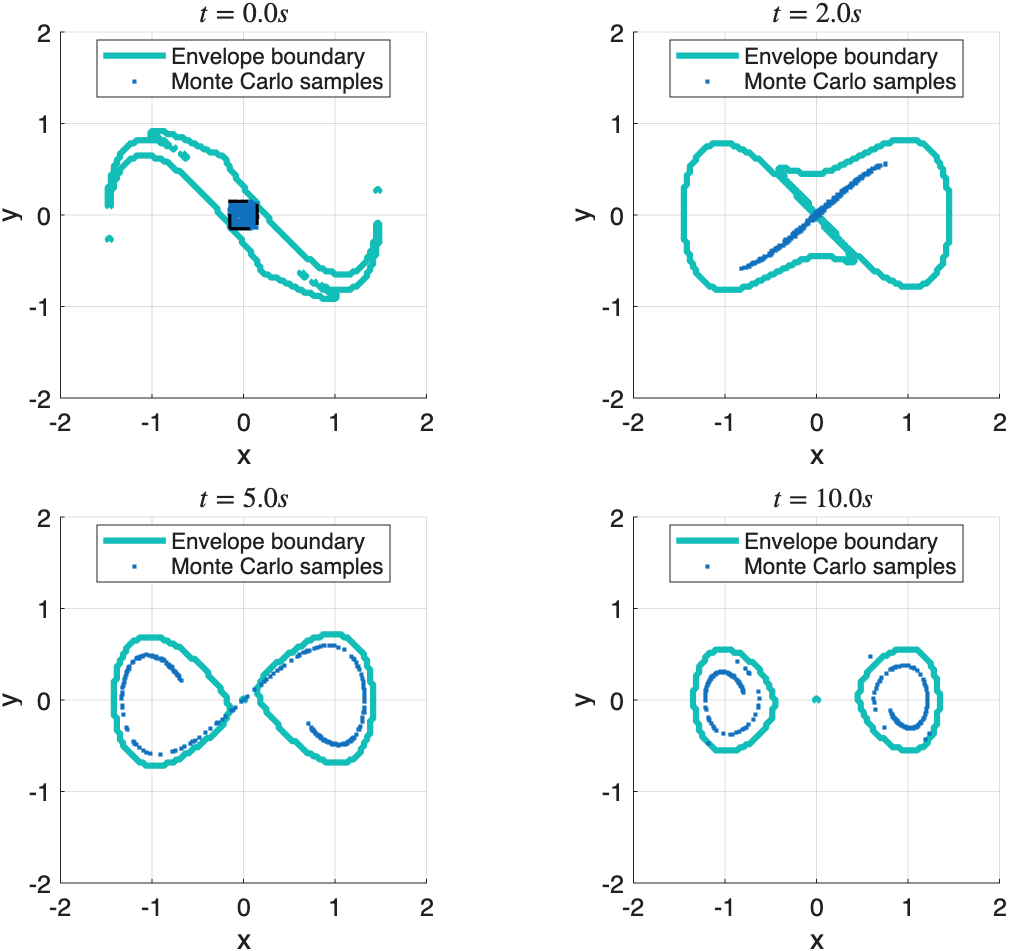}
\caption{Uncertainty propagation for the Duffing oscillator. Monte--Carlo trajectories remain within the supereigenfunction-based reachable envelope.}
\label{fig:duffing_uncertainty}
\end{figure}
\begin{figure}
\centering
\includegraphics[width=1.0\linewidth]{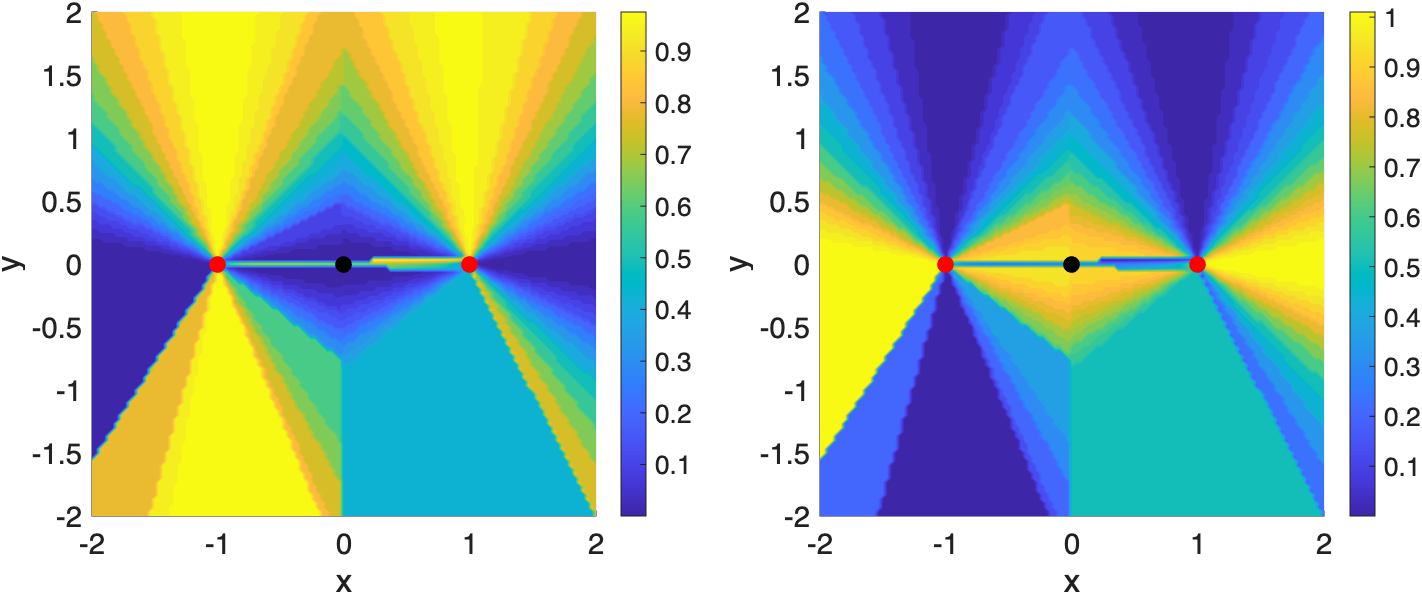}
\caption{Directional supereigenfunctions computed from the MET/QR procedure. The estimated rates are approximately $\lambda_1=-0.30$ and $\lambda_2=-0.29$.}
\label{fig:duffing_met}
\end{figure}
\subsection{Stabilization of the inverted pendulum on a cart}
\label{subsec:online_resolvent_control}

We next consider stabilization of the inverted pendulum on a cart. This example highlights the local, online nature of the resolvent-based construction: at each state, the controller requires only the value and gradient of the supereigenfunction, which are computed numerically from short rollouts of the uncontrolled system. The approach does not need the global information about the supereigenfunction and hence the approach scale gracefully with the state space dimension. The state is
$
x=\begin{bmatrix}p & \dot p & \theta & \dot\theta\end{bmatrix}^\top,
$
where \(p\) is the cart position and \(\theta\) is measured from the upright configuration. The upright equilibrium is
$
x_{\mathrm{eq}}=\begin{bmatrix}0 & 0 & 0 & 0\end{bmatrix}^\top .
$
The dynamics are
\begin{align}
\dot x_1 &= x_2,\qquad \dot x_3 = x_4, \nonumber\\
\dot x_2
&=
\frac{
u - b_x x_2 + m s \big(l x_4^2 + g c\big)
}{
{\cal D}(x)
}, \label{eq:invertedpendulum}
\\
\dot x_4
&=
\frac{
-u c + b_x x_2 c - m l x_4^2 c s - (M+m) g s - b_\theta x_4\,{\cal D}(x)
}{
l\,{\cal D}(x)
},\nonumber
\end{align}
where $s=\sin x_3$, $c=\cos x_3$, and ${\cal D}(x)=M+ms^2$. The parameters are
$
M=1.0,\;m=0.2,\; l=0.5,\;g=9.81.$
The supereigenfunction is approximated by the discrete resolvent
\begin{equation}
\varphi_\lambda(x)
\approx
\Delta t
\sum_{j=0}^{H}
e^{-\lambda j\Delta t}\,
g_p(x_j),
\label{eq:cartpole_resolvent_discrete}
\end{equation}
using uncontrolled rollouts with
$
\Delta t=0.01,\;\; H=600,\;\;\lambda=1.
$
The probe is chosen as
\begin{equation}
g_p(x)
=
b_x \dot p^{\,2} + b_\theta \dot\theta^{\,2}
+\mu \left(
\frac{k_x}{2}p^2 + k_\theta(1-\cos\theta)\right),
\label{eq:cartpole_probe}
\end{equation}
with
$
b_x=0.10,\; b_\theta=0.05,\;
k_x=1.0,\; k_\theta=5.0,\; \mu=2.0.
$
This probe combines dissipative and potential terms and is aligned with the system’s mechanical structure.
The gradient $\nabla\varphi_\lambda(x)$ is computed by centered finite differences. We first apply the QP controller in \eqref{eq:stabilizing_qp} with $
\beta=-0.5,\; u_{\min}=-10,\; u_{\max}=10.
$
From the initial condition $
x_0=(0,0,0.95\pi,0)^\top,
$
the controller stabilizes the upright equilibrium, as shown in Fig.~\ref{fig:cartpole_qp}. We also test the gradient-based controller
\begin{equation}
u
=
{\rm sat}_{[-10,10]}
\left(
-\frac12 {\cal K}_g  \varphi_\lambda(x)
\right),
\label{gradcontroller}
\end{equation}
where $g(x)$ denotes the control vector field. This controller also stabilizes the system from the same initial condition. Notably, the certificate is computed from rollout data, and only the control direction is used in feedback. A comparison of the evolution of the supereigenfunction along the system trajectory shows that it decays monotonically under the QP-based controller, whereas this monotonic decay is not observed for the gradient-based controller. This behavior is expected, as the QP formulation explicitly enforces a contracting inequality that guarantees asymptotic decay of the supereigenfunction along the closed-loop trajectory. The proposed approach for stabilizing controller design is scalable for system with large dimensional state space as the supereigenfunction are not computed globally but point-wise along the closed loop trajectory with rollouts. 

\begin{figure}
\centering
\includegraphics[width=1.0\linewidth]{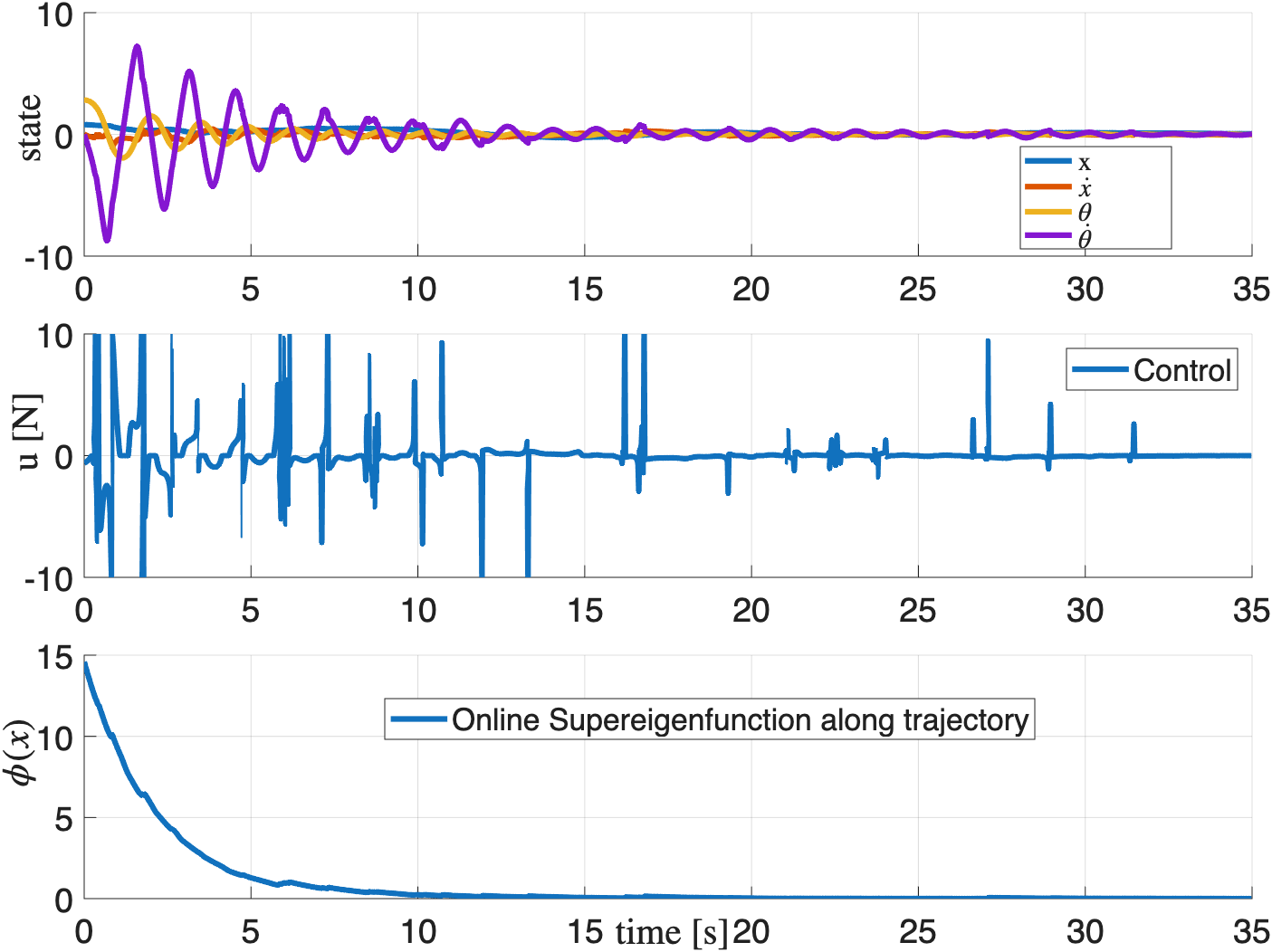}
\caption{Supereigenfunction-based stabilization of the cart-pole using the QP controller \eqref{eq:stabilizing_qp}. Top: state trajectories. Middle: control input. Bottom: decay of the certificate along the closed-loop trajectory.}
\label{fig:cartpole_qp}
\end{figure}

\begin{figure}
\centering
\includegraphics[width=1.0\linewidth]{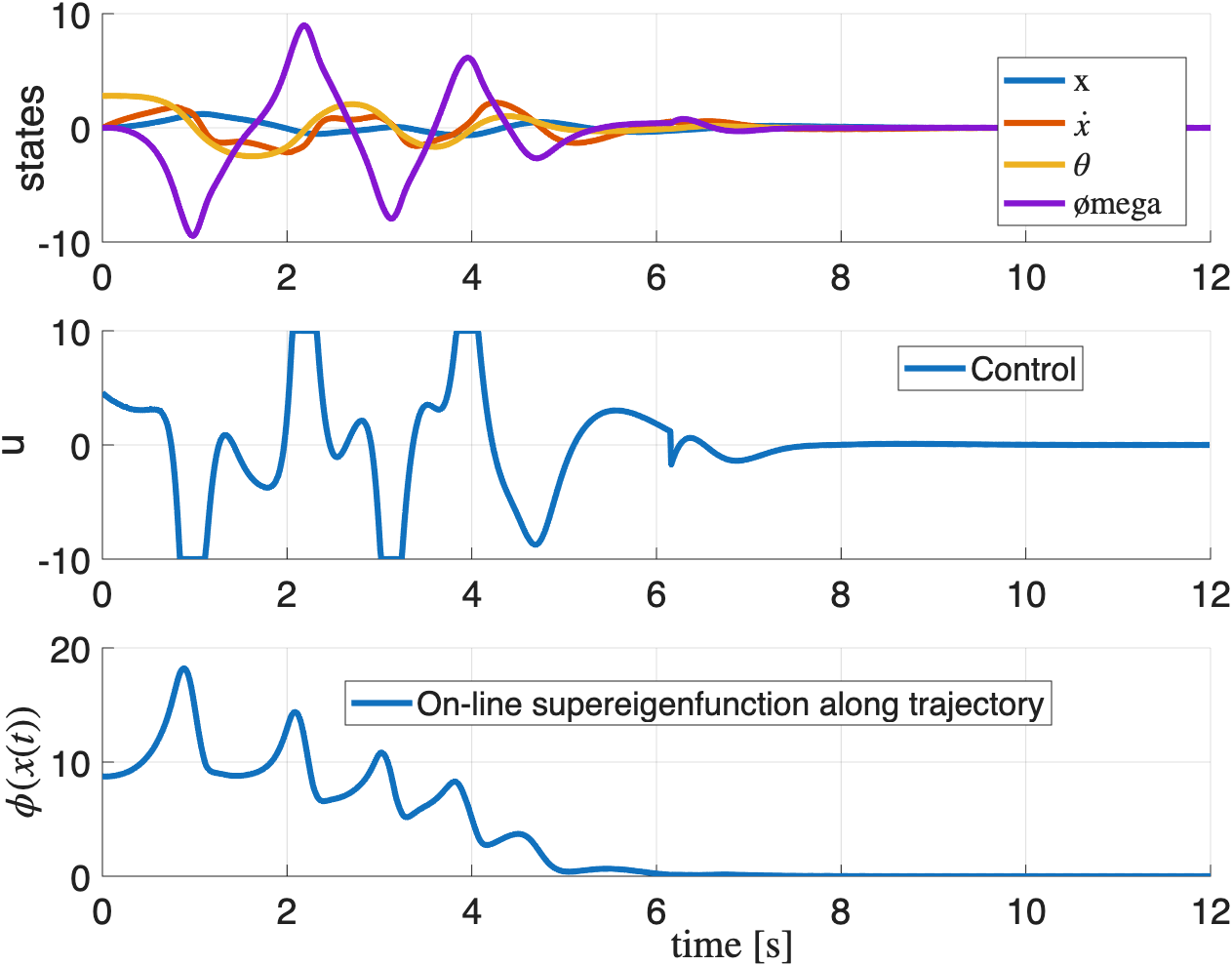}
\caption{Supereigenfunction-based stabilization of the cart-pole using the gradient controller \eqref{gradcontroller}. }
\label{fig:cartpole_grad}
\end{figure}

\subsection{4D Duffing system with obstacle avoidance}

Finally, we consider a four-dimensional system formed by two coupled Duffing oscillators:
\begin{align}
\dot{x}_1 &= y_1, \qquad
\dot{y}_1 = -\delta y_1 - \alpha (x_1-x_2) - \beta x_1^3 + u_1, \nonumber\\
\dot{x}_2 &= y_2, \qquad
\dot{y}_2 = -\delta y_2 - \alpha (x_2-x_1) - \beta x_2^3 + u_2,
\label{eq:4d_duffing}
\end{align}
with state \(z=[x_1,y_1,x_2,y_2]^\top\), control \(u=[u_1,u_2]^\top\), and parameters $
\alpha=-1,\; \beta=1,\; \delta=0.3,\; u\in[-3,3]^2.
$

\subsubsection{Nominal stabilization}

A nominal stabilizing controller is constructed from the resolvent-based certificate
\[
\phi_{\mathrm{stab}}(z)
=
\int_0^{T_1} e^{-\lambda_{\rm stab} t} g(z(t))\,dt,
\;\;
g(z)=x_1^2+y_1^2+x_2^2+y_2^2,
\]
with
$
\lambda_{\rm stab}=1.0,\; \Delta t=0.05,\; H=120,\; T_1=H\Delta t$.
The gradient of \(\phi_{\mathrm{stab}}\) is approximated by finite differences, and the nominal controller is
\[
u_{\mathrm{nom}}
=
-{\rm sat}_{[-3,3]^2}\big({\cal K}_g \phi_{\mathrm{stab}}\big).
\]

\subsubsection{Safety filtering}

Two unsafe regions are defined as spheres in ${\mathbb R}^4$:
\[
\mathcal O_i=\{z:\ \|z-c_i\|^2\le r_i^2\},\qquad i=1,2,
\]
with
$
c_1 = (1.0,0.6,-0.6,0.5)^\top,\; r_1 = 1.32,
\;
c_2 = (-1.2,-0.7,0.9,-0.4)^\top,\;r_2 = 1.5.
$
Define the margins
\[
m_i(z)=\|z-c_i\|^2-r_i^2,
\]
so that \(m_i(z)>0\) corresponds to safe operation. For each obstacle we construct a safety certificate using the probe
\[
g_i(z)=\frac{10}{1+\exp(\kappa m_i(z))},
\qquad \kappa=25,
\]
and
\[
\phi_{\mathrm{safe},i}(z)
=
\int_0^{T_2} e^{-\lambda_{\rm safe} t}\,g_i(z(t))\,dt,
\]
with $
\lambda_{\rm safe}=1.0,\; \Delta t=0.05,\; H=180,\; T_2=H\Delta t.$ At each step, the control input is obtained from the QP
\begin{align}
\min_{u\in\mathbb R^2}\quad &
\frac12 (u-u_{\mathrm{nom}})^\top R (u-u_{\mathrm{nom}}) \nonumber\\
\text{s.t.}\quad &
L_f \phi_{\mathrm{safe},i}(z)+L_g\phi_{\mathrm{safe},i}(z)\,u
\le
-\alpha_i \phi_{\mathrm{safe},i}(z),\nonumber\\
&u\in[-3,3]^2,\;\;i=1,2,
\label{eq:duffing_safety_qp}
\end{align}
with
$
\alpha_1=\alpha_2=0.1,\; R=0.5I.
$
The system is simulated for \(T=10\) seconds from $
z_0=[2.5,0.8,-2.0,0.6]^\top$.
The nominal controller drives the state toward the origin, while the safety filter modifies the control only when the trajectory approaches an obstacle. Safety is verified by
\[
\min_{t\in[0,T]} m_i(z(t))>0,\qquad i=1,2.
\]
For visualization, the trajectory is projected onto the \((x_1,x_2)\) and \((y_1,y_2)\) planes. This example illustrates simultaneous stabilization and safety enforcement in a higher-dimensional nonlinear system. It is important to emphasize that the minimum distances from the origin to obstacles 
$1$ and 
$2$ are approximately 
$0.08$ and 
$0.20$, respectively. Despite these tight safety margins, the proposed control approach successfully steers the system to the origin while maintaining safety.

\begin{figure}
\centering
\includegraphics[width=1.0\linewidth]{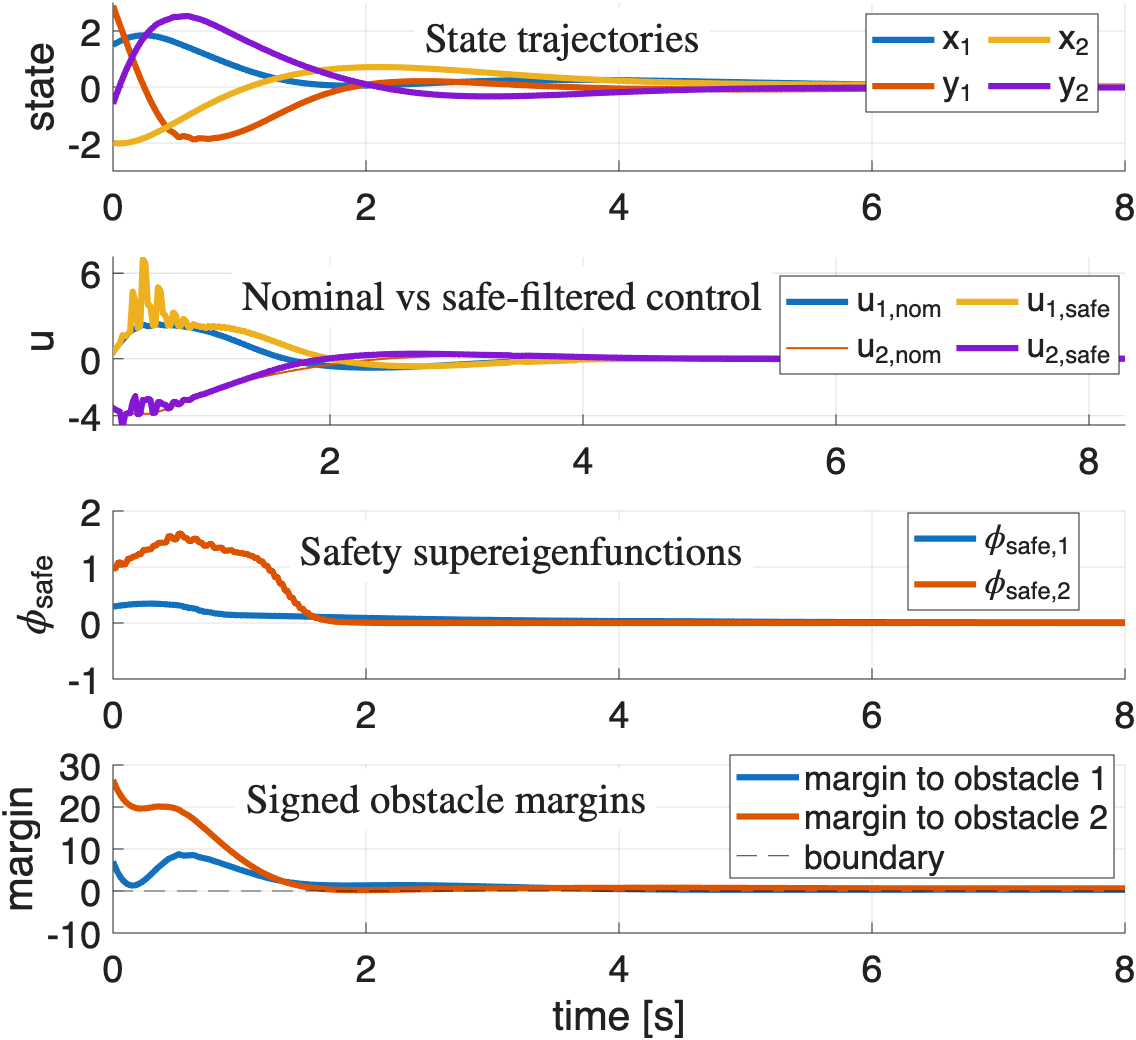}
\caption{Closed-loop trajectory of the 4D Duffing system under the nominal stabilizing controller and the safety filter. The safety constraints prevent entry into the obstacle regions while preserving convergence toward the origin.}
\label{fig:duffing4d_safe}
\end{figure}

\begin{figure}
\centering
\includegraphics[width=1.0\linewidth]{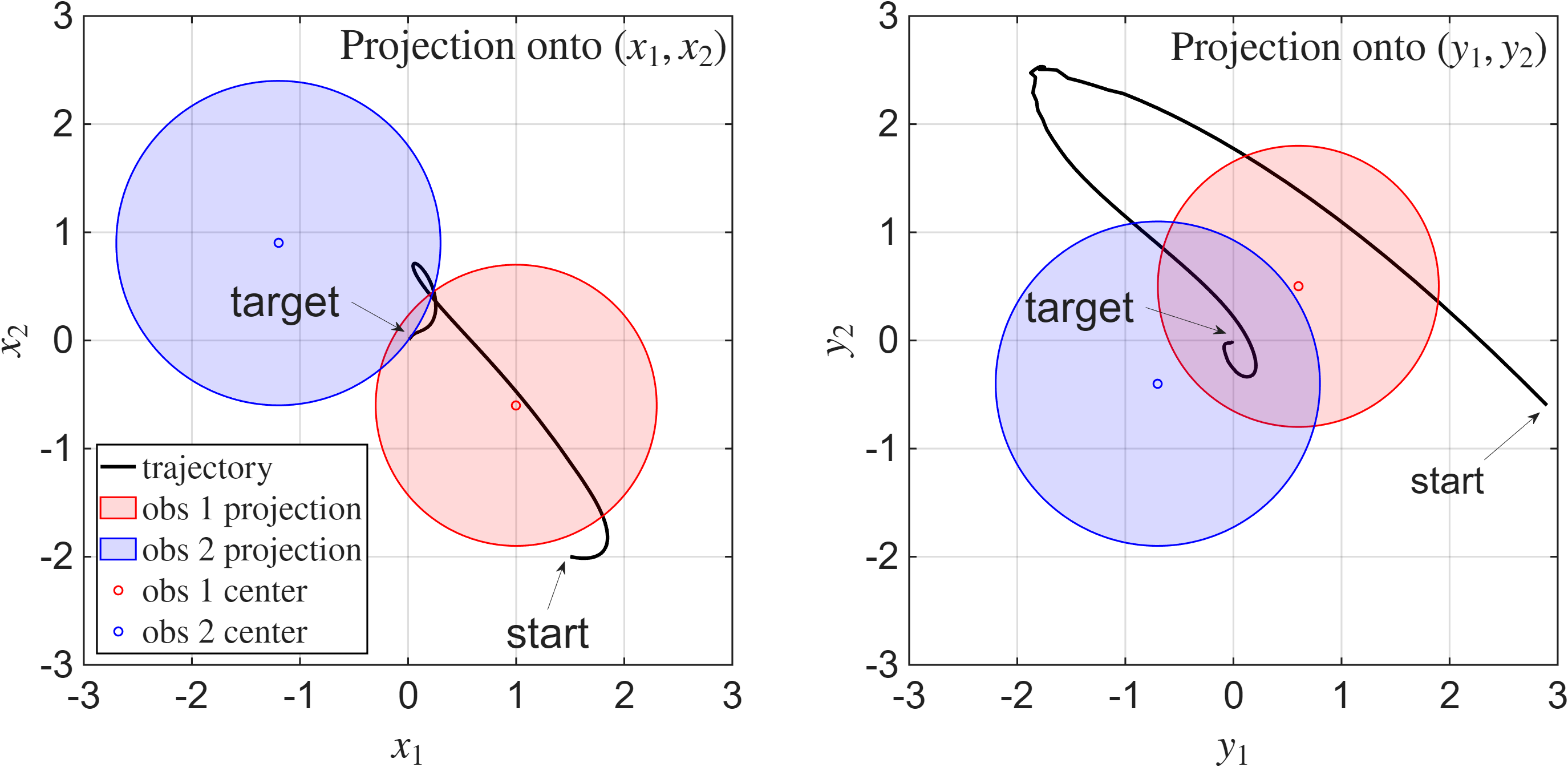}
\caption{Projection of the 4D closed-loop trajectory onto selected coordinate planes. The projected obstacle regions are avoided throughout the motion.}
\label{fig:duffing4d_projection}
\end{figure}

\subsection{Synchronization of a Forced Duffing Oscillator}

We consider a leader--follower pair of Duffing oscillators. The leader is driven by a sinusoidal input,
\begin{align}
\dot x_L &= y_L, \nonumber\\
\dot y_L &= -\delta y_L-\alpha x_L-\beta x_L^3 + A_L\sin(\omega_L t),
\end{align}
while the follower evolves as
\begin{align}
\dot x_F &= y_F, \nonumber\\
\dot y_F &= -\delta y_F-\alpha x_F-\beta x_F^3 + u.
\end{align}
The parameters are $
\alpha=-1,\; \beta=1,\; \delta=0.2,\quad
A_L=0.3,\; \omega_L=1, $
with simulation horizon \(T=120\) s and step size \(\Delta t=0.01\). The initial conditions are chosen to be
$
z_L(0)=\begin{bmatrix}0.1 & 0.2\end{bmatrix}^\top,
\;
z_F(0)=\begin{bmatrix}-0.1 & -0.3\end{bmatrix}^\top 
$
and corresponds to initializing the system in two different potential wells of the oscillator. 
Let \(e=z_F-z_L\) denote the synchronization error. A finite-horizon Gramian metric is computed along the forced leader trajectory,
\[
M_{T_g}(x)=\int_0^{T_g} e^{-2\lambda \tau}Y(t,x)^\top QY(t,x)\,dt,
\]
where $Y(t,x)=Ds^L_{t}(x)$ is the Jacobian evaluated along the trajectory of the leader system. The 
parameters are
$
\Delta t_g=0.02,\; H_g=500,\; \lambda=5,\;Q=I_2.
$
The corresponding tangent-space certificate is
$
V_\delta(t)=e(t)^\top M_{T_g}(x)(t)e(t).
$ The follower control is synthesized using a scalar quadratic program that enforces contraction of the certificate:
\begin{align}
\min_u \quad & \tfrac12 (u-u_{\rm nom})^2 \\
\text{s.t.} \quad &
2e^\top M_{T_g}(x)\big(f_F(z_F)+Bu-f_L(z_L,t)\big)\nonumber\\
&\le -\gamma\,e^\top M_{T_g}(x) e,\nonumber
\end{align}
where \(B=[0\ \ 1]^\top\), \(u_{\rm nom}=0\), \(\gamma=2\), and \(u\in[-30,30]\). This formulation yields a minimum-intervention control law that enforces decay of the tangent-space certificate. Simulation results show that the follower successfully synchronizes with the oscillatory leader, while the certificate \(V_\delta(t)\) decreases along the trajectory, demonstrating contraction-based synchronization using Gramian supereigenfunction certificates.

\begin{figure}
\centering
\includegraphics[width=1.0\linewidth]{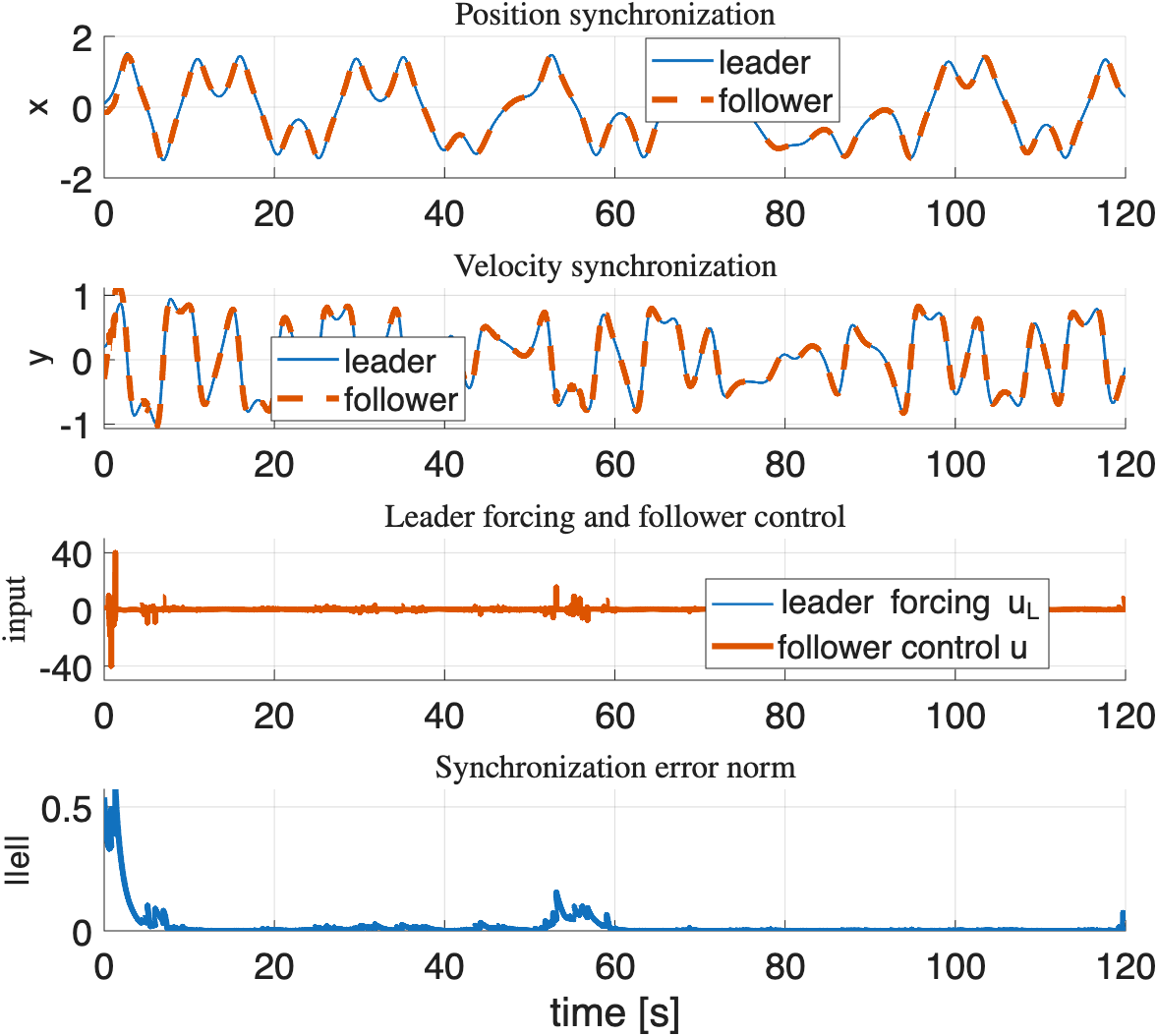}
\caption{Synchronization in leader-follower oscillators.}
\label{fig:synchronization_state}
\end{figure}

\begin{figure}
\centering
\includegraphics[width=1.0\linewidth]{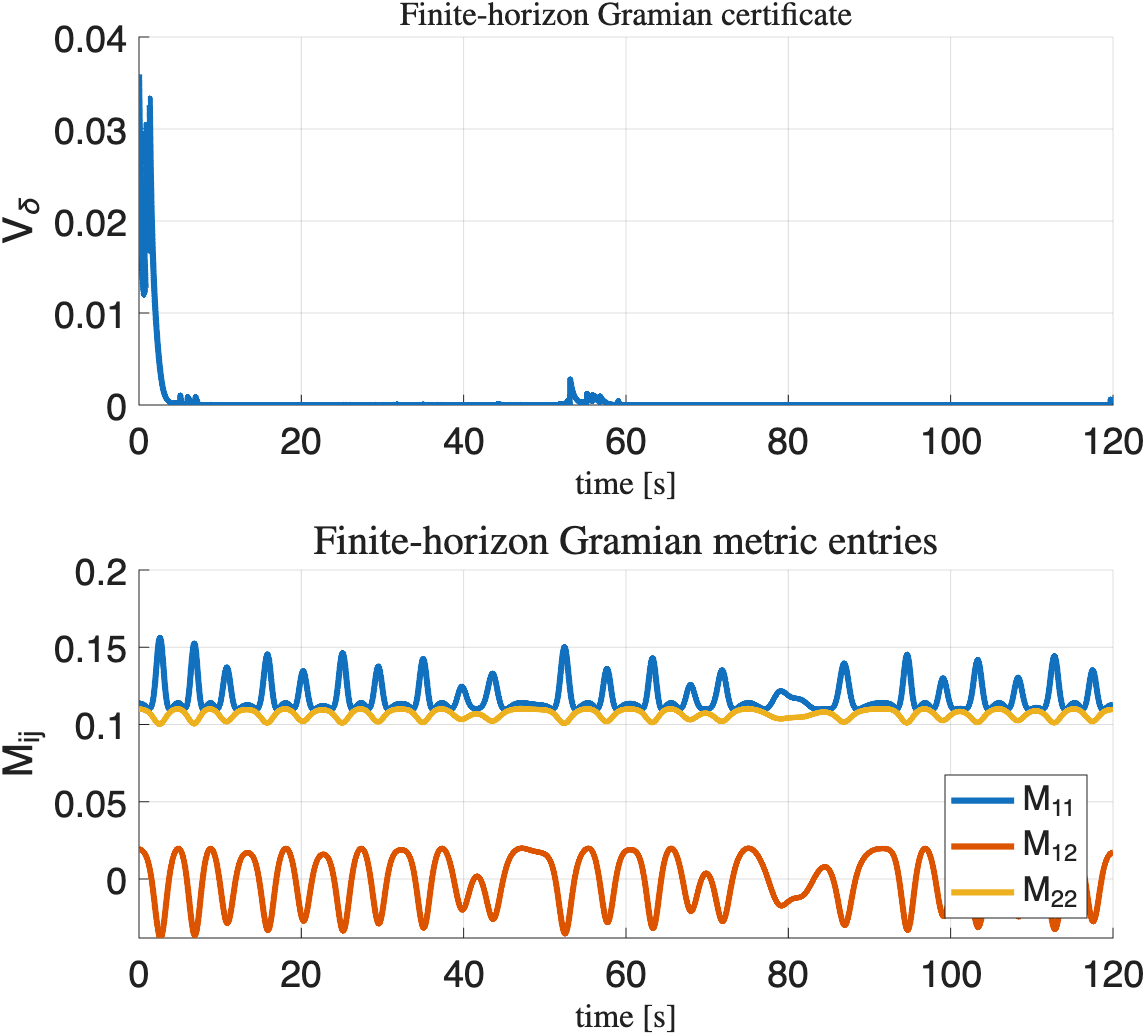}
\caption{Evolution of supereigenfunction as contraction certificate.}
\label{fig:synchronization_certificate}
\end{figure}

% These four examples illustrate the main roles of the proposed framework: inequality-based modeling and uncertainty propagation in the Duffing oscillator, rollout-based stabilizing control for the cart-pole, contraction-based synchronization in a leader–follower Duffing system, and integrated stabilization and safety filtering in a higher-dimensional nonlinear system.

\section{Conclusion and Future Direction}
This paper introduced \emph{supereigenfunctions} as an inequality-based generalization of Koopman eigenfunctions, providing an operator-theoretic framework for analyzing and controlling nonlinear systems through bounds and certificates. The key contribution is a shift in perspective: Koopman methods are used not to construct predictive models, but to construct \emph{certificates} such as Lyapunov functions, Hamilton--Jacobi value functions, and safety constraints.
We developed directional supereigenfunctions for contraction and incremental stability, and resolvent/Gramian constructions that enable computation directly from trajectory data. These results establish a bridge between Koopman operator theory and inequality-based control methods, while providing a scalable, data-driven alternative to PDE and SOS-based approaches.
% Future work includes deeper connections with Hamilton--Jacobi theory and stochastic systems, as well as further development of data-driven algorithms and integration with learning-based control frameworks.
% Overall, this work lays the foundation for an inequality-based Koopman framework for certificate construction and data-driven control.

% \bibliographystyle{ieeetr}
% \bibliography{references}
\bibliographystyle{IEEEtran}
\bibliography{references2_arxiv}
 \begin{IEEEbiography}
[{\includegraphics[width=1in,height=1.25in,clip,keepaspectratio]{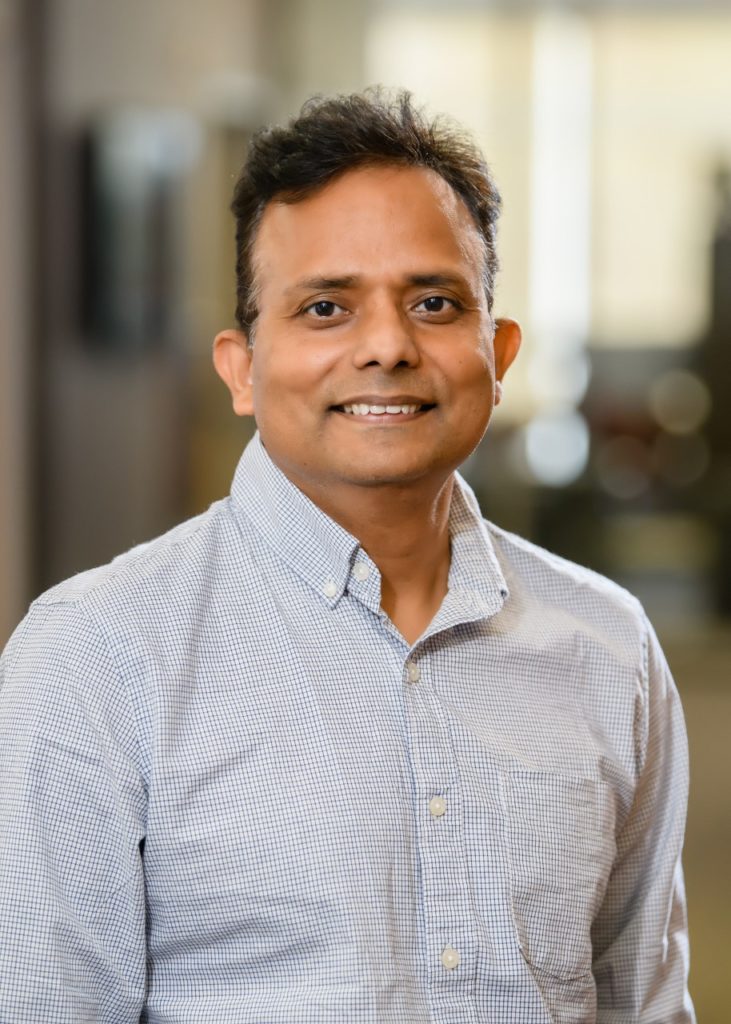}}]%
{Umesh Vaidya}(M’07, SM'19)  received the Ph.D. degree in mechanical engineering from the University of California at Santa Barbara, Santa Barbara, CA, in
2004. He was a Research Engineer at the United Technologies Research Center (UTRC), East Hartford, CT, USA. He is currently a professor in the Department of Mechanical Engineering, Clemson University, S.C., USA. Before joining Clemson University in 2019, and since 2006, he was a faculty with the department of Electrical and Computer Engineering at Iowa State University. He is the recipient of 2012 National Science Foundation CAREER award. His current research interests include dynamical systems and control theory with applications to power grid and robotics. 
\end{IEEEbiography}
\end{document}